\documentclass[12pt,english,letterpaper]{article}

\usepackage[english]{babel}
\usepackage{amsmath,amsthm,amssymb,amsfonts,amscd,amstext}
\usepackage{mathrsfs}
\usepackage{bm}
\usepackage{xcolor}
\usepackage{graphicx}
\usepackage{enumerate}
\usepackage[title]{appendix}
\usepackage{hyperref}
\hypersetup{
    colorlinks=true,
    citecolor=blue
}

\allowdisplaybreaks

\newtheorem{theorem}{Theorem}
\newtheorem{corollary}[theorem]{Corollary}
\newtheorem{definition}[theorem]{Definition}
\newtheorem{example}[theorem]{Example}
\newtheorem{lemma}[theorem]{Lemma}

\newtheorem{proposition}[theorem]{Proposition}
\newtheorem{remark}[theorem]{Remark}

\newcommand{\norm}[1]{\lVert#1\rVert}
\newcommand{\abs}[1]{\left|#1\right|}

\renewcommand{\natural}{\mathbb{N}}
\newcommand{\real}{\mathbb{R}}
\newcommand{\complex}{\mathbb{C}}
\newcommand{\koopman}{K}
\newcommand{\ruelle}{L}
\newcommand{\mult}{M}

\newcommand{\dirac}{\mathcal{D}}

\newcommand{\cont}{\mathcal{C}}

\newcommand{\lp}[1]{\mathscr{L}^{#1}}

\newcommand{\defn}{:=}
\renewcommand{\d}{\operatorname{d} \!}

\begin{document}

\title{The Dirac operator for the pair of  Ruelle and Koopman operators, and a generalized Boson formalism }

\author{William  M. M. Braucks and  Artur O. Lopes}

\maketitle

\begin{abstract}
{\footnotesize 
Denote by $\bm{\mu}$ the maximal entropy measure for the shift map \(\sigma\) acting on $\Omega = \{0, 1\}^\mathbb{N}$,
        by $\ruelle$ the associated Ruelle operator and by $\koopman = \ruelle^{\dagger}$ the Koopman operator, both acting 
        on $\lp{2}(\bm{\mu})$. The Ruelle-Koopman pair can determine a generalized boson system  in the sense of \cite{Kuo}.    Here
$2^{-\frac{1}{2}} K$ plays the role of the creation operator
        and $   2^{-\frac{1}{2}} L$  is the annihilation  operator.  We show that $[L,K]$ is the projection on the kernel of $L.$ In $C^*$-algebras the Dirac operator  $\mathcal{D}$ represents  derivative.
Akin to this point of view we introduce a dynamically defined Dirac operator  $\mathcal{D}$ associated with the Ruelle-Koopman pair and  a  representation $\pi$.  Given  a continuous function  $f$, denote by $\mult_{f}$  the operator $ g \to M_f(g)=f\, g.$ Among other dynamical relations we get

\smallskip

$\norm{\left[ \mathcal{D} , \pi ( \mult_{f} ) \right]} = \sup_{x \in \Omega} \sqrt{\frac{\abs{f(x) - f(0x)}^{2}}{2} + \frac{\abs{f(x) - f(1x)}^{2}}{2}} = \abs{\sqrt{\ruelle \abs{\koopman f - f}^{2}}}_{\infty}$
       \smallskip
       
\noindent which concerns a  form of discrete-time mean backward derivative. 
We also derive   an inequality for the   discrete-time  forward derivative $f \circ \sigma -f$:
\smallskip

$\,\,\,\,\,  \,\,\,\, |f \circ \sigma -f |_{\infty} = \abs{\koopman f - f}_{\infty}  \geq \norm{\left[ \mathcal{D} , \pi ( \mult_{f} ) \right]} \geq \abs{f - \ruelle f}_{\infty}.$
\smallskip

\noindent  Moreover, we get $\|\, \left[\mathcal{D} ,\pi(\koopman \ruelle)\right] \,\|=1$. The Number operator is $\frac{1}{\sqrt{2}}K \frac{1}{\sqrt{2}} L.$ The Connes distance  requires to ask when an operator $A$ satisfies the inequality  $\|\, \left[\mathcal{D} ,\pi(A)\right] \,\|\leq 1$; the Lipschtiz constant of $A$  smaller than $1$.}

\end{abstract}

\section{Introduction} \label{iintoc}

In the physics of elementary particles, the bosons (the same is true for fermions) are particles that
are indistinguishable from each other. For example, every electron (which is a fermion particle) is
exactly the same as every other electron. In this way, given a state with two electrons, you can exchange
the two electrons and this will not change anything physically observable from that state. For bosons,
the number of indistinguishable particles can range in the set of natural numbers $0, 1, 2, \cdots, n, \cdots$;
several identical bosons can simultaneously occupy the same quantum state. One of the main issues on
the topic is the understanding of the statistics of a collection of non-interacting identical particles
that may occupy a set of available discrete energy states at thermodynamic equilibrium. Half of the elementary  particles
of the universe are bosons (according to the table of the Standard model as mentioned in \cite{Bra}). 

For fermions, the Pauli exclusion principle claims: only one fermion can occupy a particular quantum
state at a given time. A very detailed study of fermions appears in \cite{BruW}. Our focus here will be on the boson formalism.

In \cite{Kuo} the authors analyze the so-called generalized boson systems where the classical canonical
commutation relation (CCR) is not true; they extend boson sampling protocols to a larger class of quantum
systems, that include, among others, interacting bosons. The generalized point of view of \cite{Kuo}
will be in significant consonance with our dynamical setting (see \eqref{ELV1}, \eqref{lalau},
and more details in Section \ref{bobo}). In the more abstract and general point of view of \cite{Kuo},
$b$ and $b^\dag$ denote, respectively, the annihilation and creation operator. The CCR would correspond to:
\begin{equation} \label{tor0} 
    I= \left[b,b^\dag\right] = b\,b^\dag - b^\dag\, b \text,
\end{equation}
which in this general case is not always true. We will call the case where \eqref{tor0} applies classic,
to distinguish it from the generalized case we will deal with here.

General references for classical bosons and fermions are \cite{Arai}, \cite{Schwabl}, \cite{Potts}
and \cite{Hall}. A function $n \to f(n)$ plays an important role in \cite{Kuo}, and the main relationship
would be:
\begin{equation} \label{sor}
    b^\dag |n \rangle = \frac{f(n+1)}{f(n)}|n +1 \rangle \,\, \,\text{ and} \,\,\, b|n \rangle = \frac{f(n)}{f(n-1)}|n-1 \rangle \text,
\end{equation}
where $|n\rangle$ describe the state with $n$-particles, $n = 0, 1, 2, \cdots$ (see page 043096-5 in \cite{Kuo}).

The state $|0\rangle$, which satisfies:
\begin{equation} \label{vacc0}
    b( |0\rangle ) = 0 \text,
\end{equation}
is called the vacuum.

For the classical case $f(n)= \sqrt{n!}$, for the boson pair  $f(n)= \sqrt{(2\,n)!}$ (see page 043096-2
in \cite{Kuo} and also \cite{Cian}), and here will consider the case where $f(n) = 2^{-\frac{n}{2}}$ (see
\eqref{ELV1} and \ref{lalau}).

In the classical case, that is $f(k)= \sqrt{k!}$, the operators $\mathcal{C}= b^\dag$ and $\mathcal{A}=b$
act on an infinite-dimensional Hilbert space and the  elements $|k\rangle$, $k = 0, 1, 2, \cdots$,
are eigenfunctions for $\mathcal{C} \, \mathcal{A}$, associated to eigenvalues which range in the set of
natural numbers. More precisely:
\begin{equation} \label{mnoto378}
    \mathcal{C} \, \mathcal{A} ( |k\rangle)= k\, |k\rangle= (\frac{f(k)}{f(k-1)})^2\, |k\rangle \text.
\end{equation}

In this case, the self-adjoint operator:
\begin{equation} \label{mnoto3781}
    \mathcal{C} \, \mathcal{A}
\end{equation}
is called the (classical) {\it number operator}. The terminology {\it number operator} is due to such
a property: the count of particles at a given state (see Section 5 in \cite{Arai}).

In our setting, expression \eqref{ELV2} corresponds to \eqref{mnoto378}.

The quantization of the harmonic oscillator can be put in a form that satisfies \eqref{tor0} and also
the above expression \eqref{mnoto378} (see Section 11 in \cite{Hall}). In this case, $|n\rangle$ is the
\(n\)th eigenfunction of the corresponding quantized Hamiltonian operator ${\bf H}$ acting on functions
on $\lp{2}(\d x)$ defined on the real line (see details in  Example \ref{lour} in section \ref{bobo}).
In Quantum Mechanics the self-adjoint operators are the observables and their eigenvalues correspond
to the real values that can be observed when measuring (see \cite{LQM}). The eigenvalues of the number operator of a
fermion can be just $0$ or $1$. The value $1$ means occupied and $0$ means not occupied. The eigenvalues
of the number operator of a boson can be any natural number.

We refer the reader to the beginning of Section II (or (A3)) in \cite{Kuo} (which deals with the
generalized case), where denoting by $|0\rangle$ the vacuum, one gets:
\begin{equation} \label{m37}
    (b^\dag)^k |0\rangle = f(k)\, |k\rangle \text.
\end{equation}

In the case of the harmonic oscillator:
\begin{equation} \label{mo37}
    \mathcal{C}^k |0\rangle =\sqrt{k!} \,|k\rangle \text.
\end{equation}

In this case, $|0\rangle$ corresponds to the eigenfunction associated with the smallest eigenvalue
of the Hamiltonian operator ${\bf H}$, and:
\begin{equation} \label{vvac1}
    \mathcal{A} ( |0\rangle)=0 \text.
\end{equation}
For this reason $|0\rangle$ is called the vacuum (eigenfunction) for the harmonic oscillator (see \eqref{vacc0}).

The classical CCR would be given by $\left[\mathcal{A}, \mathcal{C}\right]=I$, but the corresponding relation
is not true in our setting (see \eqref{poet2}); however, a generalized CCR, in the sense of (2) and (A5) in \cite{Kuo}, 
is true in our setting. It is given in \eqref{leler34} and Proposition \ref{lalau}.

More precisely, (2) in \cite{Kuo} requires the existence of a naturally defined function
$F:\mathbb{N} \to  \mathbb{R}$ (to be determined according to the specific quantum problem), such that
by definition:
\begin{equation} \label{leler2}
    \left[ b, b^\dag \right]:= \sum_n F(n) |n\rangle\,\langle n| \text.
\end{equation}

The classical CCR corresponds to taking $F \equiv 1$.

We will consider here the Hilbert space $\lp{2}(\bm{\mu})$, when $\bm{\mu}$ is the measure of maximal entropy for
the action of the shift $\sigma: \Omega \to \Omega$, where $\Omega= \{0,1\}^\mathbb{N}$.

The Ruelle operator $\ruelle: \lp{2}(\bm{\mu}) \to \lp{2}(\bm{\mu})$ is defined for a continuous function $f_1$,
by $\ruelle(f_1)=f_2$, if for all $x\in \{0,1\}^\mathbb{N}$,  $f_2(x)= \frac{1}{2} (f_1 (0,x) + f_1(1,x))$.

In Section \ref{bobo}, we will show that it is natural to call $B = 2^{-\frac{1}{2}} \ruelle$ the annihilation operator, where $\ruelle: \lp{2}(\bm{\mu}) \to \lp{2}(\bm{\mu})$ is the
Ruelle operator. It is known that  $\koopman = \ruelle^{\dagger}$ is the Koopman operator (see \eqref{lui17} and \eqref{elui1591}). $B^{\dagger} = 2^{-\frac{1}{2}} \ruelle^{\dagger}=  2^{-\frac{1}{2}} K$ will  be called the creation operator. 
         We will show that $[L,K]$ is the projection on the kernel of $L$  (see \eqref{leler34} and  Lemma \ref{Boscf}).  $B^{\dagger} B $ will be called the number operator (see \eqref{nu34}).

In our framework, the boson formalism manifests itself through the expressions in Propositions \ref{esta} and \ref{kesta}.
We will elaborate on that.

    Let \(W\) be the set of words of finite length on the symbols $0$ and $1$. Typical words are denoted \(w, u, v \in W\).
    \(\ell(w)\) is the length of \(w\). Given \(k \in \natural\), we denote the set of words of length at most \(k\) by
    \(W_{k}\), and the set of words of length exactly \(k\) by \(\hat{W}_{k} \cong \left\{0,1\right\}^{k}\). There is an
    empty word \(\varepsilon \in W\), and its length is \(\ell(\varepsilon) = 0\). The concatenation of the words \(u, v \in W\) is denoted by
    \(uv \in W\). Similarly, we concatenate a word \(w \in W\) to the left of a sequence \(x \in \Omega\) by
    \(wx = (w_{1}, w_{2}, \cdots, w_{\ell(w)}, x_{1}, x_{2}, \cdots) \in \Omega\).

 For a given finite word
    \(w \in W\) the cylinder set \([w]\) is defined by:
    \begin{align*}
        [w] \defn \left\{x \in \Omega \mid x_{1} = w_{1}, x_{2} = w_{2}, \cdots, x_{\ell(w)} = w_{\ell(w)}\right\} \text.
    \end{align*}
    Notice \([\varepsilon] = \Omega\).
    
For any \(w \in W\), let \(\chi_{[w]}\) denote the characteristic function of the cylinder \([w]\). Recall that the
    cylinder is clopen so \(\chi_{[w]}\) is continuous. Consequently, any finite linear combination of these, \textit{i.e.}
    any simple function, is also continuous. Notice \(\Omega\) is compact, so that \(\cont(\Omega) \subseteq \lp{2}(\bm{\mu})\)\footnote{Of
    course it is not \(\cont(\Omega)\) that is contained in \(\lp{2}(\bm{\mu})\), but its image via the \(\lp{2}(\bm{\mu})\) quotient map.}.

    Consider the index set \(W^{\ast}\) of nonempty words. The Hilbert space \(\lp{2}(\bm{\mu})\) has an orthonormal basis given by:
    \begin{align*}
        \left\{e_{\varepsilon}^{1} \defn 2^{\frac{1}{2}} \chi_{[1]}, e_{\varepsilon}^{0} \defn - 2^{\frac{1}{2}} \chi_{[0]}\right\} \cup \left\{e_{w} \defn 2^{\frac{\ell(w)}{2}} \left(\chi_{[w1]} - \chi_{[w0]}\right) \mid w \in W^{\ast}\right\} \text.
    \end{align*}
    
    \smallskip
    
    This result was derived in  \cite{CHLS} following a general formulation described in \cite{KS}.
    
    \smallskip
    
    With respect to this basis, the Ruelle and Koopman operators are characterized by:
    \begin{align*}
        \koopman(e_{w}) & = 2^{-\frac{1}{2}} \left( e_{0w} + e_{1w} \right) \\
        \ruelle(e_{w}) & = 2^{-\frac{1}{2}} e_{\sigma(w)} \\
        \koopman(e_{\varepsilon}^{0}) & = -2^{\frac{1}{2}} \left( \chi_{[00]} + \chi_{[10]} \right) \\
        \koopman(e_{\varepsilon}^{1}) & = 2^{\frac{1}{2}} \left( \chi_{[01]} + \chi_{[11]} \right) \\
        \ruelle(e_{\varepsilon}^{0}) & = -2^{-\frac{1}{2}} \\
        \ruelle(e_{\varepsilon}^{1}) & = 2^{-\frac{1}{2}} \\
        \koopman(2^{\frac{1}{2}}(e_{\varepsilon}^{0} + e_{\varepsilon}^{1})) & = 2^{- \frac{1}{2}} \left( e_{0} + e_{1} \right) \\
        \ruelle(2^{\frac{1}{2}}(e_{\varepsilon}^{0} + e_{\varepsilon}^{1})) & = 0 \text.
    \end{align*}

    We set  \(B \defn 2^{-\frac{1}{2}} \ruelle\) so that \(B^{\dagger} = 2^{-\frac{1}{2}} \koopman\).
Notice that if we write \(| n \rangle \defn 2^{-\frac{n}{2}} \sum_{\ell(w) = n} e_{w}\) and
    \(| 0 \rangle \defn 2^{-\frac{1}{2}}(e_{\epsilon}^{0} + e_{\epsilon}^{1})\), so that \(\langle n , n \rangle = \langle 0 , 0 \rangle = 1\),
    we get:
    \begin{equation} \label{porty}
        B^{\dagger} | n \rangle = 2^{-\frac{1}{2}} | n + 1 \rangle \text{, and} \,\,
        B | n \rangle  = 2^{-\frac{1}{2}} | n - 1 \rangle \text;
    \end{equation}
    and consequently, for $ n \geq 1$ it follows that
$B^{\dagger n} | 0 \rangle  = 2^{-\frac{n}{2}} | n \rangle$.

From \eqref{porty}, in the generalized boson point of view,  we can say that 
$2^{-\frac{1}{2}} K$ plays the role of the creation operator
and $   2^{-\frac{1}{2}} L$   plays the role of the  annihilation  operator.

In our setting $| 0 \rangle$ corresponds to the vacuum.

Moreover,
    \begin{equation} \label{nu34}    B^{\dagger} B | n \rangle = 2^{-1} | n \rangle.
    \end{equation}

    Now, we define \(|n, w \rangle \defn 2^{-\frac{n+1}{2}} \left( \sum_{\ell(u) = n} e_{u0w} - \sum_{\ell(u) = n} e_{u1w} \right)\)
    and \linebreak \(|0, w \rangle \defn 2^{-\frac{1}{2}} \left( e_{0w} - e_{1w} \right)\). Then:
    \begin{align*}
        B^{\dagger}|n, w \rangle & = 2^{-\frac{1}{2}} |n+1, w \rangle \\
        B|n, w \rangle & = 2^{-\frac{1}{2}} |n-1, w \rangle \\
        B|0, w \rangle & = 0 \text.
    \end{align*}

    The family of elements of the form
    $$\sqrt{2}  \left(  e_{0w} -  e_{1w} \right)$$
    determines an orthonormal Hilbert basis for  the Kernel of the Ruelle operator $L$ (see \cite{LR1}).

    Considering the index set \(W^{\star} \defn \left\{\star\right\} \cup W\), define \(|n, \star \rangle \defn |n \rangle\).
    The family \(\left\{|0, w\rangle \mid w \in W^{\star}\right\}\) is an orthonormal basis for the kernel of the Ruelle
    operator. Also, for \(n \geq 1\):
    \begin{align*}
        B^{\dagger\, n} |0, w \rangle & = 2^{-\frac{n}{2}} |n, w \rangle \text{, and:} \\
        B^{\dagger} B |n, w \rangle & = 2^{-1} |n, w \rangle \text.
    \end{align*}
    Finally, by Wold's decomposition theorem,
    \begin{align*}
        \left\{\chi_{[\varepsilon]} = 1\right\} \cup \left\{|n, w\rangle \mid n \in \natural, w \in W^{\star}\right\}
    \end{align*}
    is an orthonormal basis for \(\lp{2}(\bm{\mu})\).

We will present later (see Proposition \ref{lalau}) a generalized CCR relation of the form:
\begin{equation} \label{leler34}
  \left[ B, B^{\dagger} \right] = 2^{- 1} P_{\ker \ruelle} = 2^{-1} \sum_{w \in W^{\star}} |0, w\rangle \langle 0, w|, 
\end{equation}
where $P_{\ker \ruelle}$ is the projection on the Kernel of the Ruelle operator.
 
 In the sense of \cite{Kuo}, this 
would mean to take $F(0) = 2^{-1}$ and $F(n) = 0,$ for all \(n \geq 1\). The  expression         \eqref{leler34} is our version of \eqref{leler2}.

For the case of classical  fermions, the  Canonical Anticommutative Relation (CAR) should be valid,
and this  would correspond to:
\begin{equation} \label{tor}
    I = \{b,b^\dag\} = b\,b^\dag + b^\dag\, b \text,
\end{equation}
which in the general case is not always true. In our setting, we briefly  mention  the CAR for fermions in \eqref{feo1}.
Denote $\hat{f}= \frac{1}{\sqrt{2}}\ruelle$ and $\hat{f}^\dag= \frac{1}{\sqrt{2}}\koopman$,
 and by  $\mathcal{F}$ the space  of  functions
$\phi$ in $\lp{2}(\bm{\mu})$  that do not depend on the first coordinate.

Is it natural to consider bounded operators acting on the $\mathcal{F}$. Then, under such restriction, we will show (see Section \ref{bobo}) in such $C^*$-algebra the CAR:
\begin{equation} 
  \{ \hat{f},  \hat{f}^\dag\}= I.
\end{equation}

Concerning all the above claims, in Section \ref{bobo} we will present explicit computations in order to get the  main
results we just described; summarized in Propositions \ref{ELV} and \ref{lalau}.

In our generalized dynamical boson setting we are interested, among other things, in the Number operator
$B^{\dagger} B$ which is our version of \eqref{mnoto3781}; In \eqref{plis} we introduce a Dirac operator $\mathcal{D}$,
a representation $\pi$ and we introduce a special spectral triple (see \eqref{wer777} and Definition
\ref{def-spec-triples}).

In $C^*$-Algebras the operator $\mathcal{D}$ plays the role of a derivative of self-adjoint operators.

 One of our main motivation here was to show that $\|\, \left[\mathcal{D} ,\pi(\koopman \ruelle)\right] \,\|=1$
(see Section \ref{pos}). In the {\it sense of Connes} (see \cite{cfrconnes}, \cite{Connes2}), this would
{\it mean} here that the self-adjoint operator Number operator $B^{\dagger} B$ has Lipschitz
constant equal to $2^{-1}$ (see Remark \ref{Conm}). We leave for future investigation the question of whether our mathematical formalism eventually corresponds to any application in physics.

We point out that here we will introduce a diagonal representation, which is a natural generalization
of the setting in \cite{LCH} (a finite-dimensional non-dynamical case), where the authors were able
to compute explicitly the Connes spectral distance between one-qubit states.

In the second part of the paper (see Section \ref{didi}) we analyze properties of the dynamical Dirac operator that we will  introduce.  Motivated
by \cite{LCH}, which deals with a finite-dimensional case, we consider a certain Dirac operator $\mathcal{D}$
and a special diagonal representation $\pi$ (see (5) in \cite{LCH} for the case of  qubits). Here we will set:
\begin{equation} \label{plis}
    \mathcal{D} = \left(\begin{array}{cc}
0 &  \koopman\\
\ruelle& 0
\end{array}
\right) \text.
\end{equation}

Different kinds of  (dynamical) diagonal representations were considered in  \cite{Sha1}, \cite{Sha2}, and in Section 6 in \cite{CHLS}.

\begin{definition} \label{def-spec-triples}
  A spectral triple is an ordered triple $(\mathcal{A},\mathcal{H},D)$, where:
\begin{enumerate}
\item
  $\mathcal{H}$ is a Hilbert space;

\item
  $\mathcal{A}$ is a $C^{*}$-algebra,  $\pi$ is a representation, where for each $a \in \mathcal{A}$, we can associate a
 bounded linear operator $\pi({a}) : \mathcal{H} \to \mathcal{H}$;

\item
 $D$ is an essentially self-adjoint unbounded linear operator on $\mathcal{H}$, such that
 $\{a \in \mathcal{A} \mid \norm{\left[D,\pi(a)\right]} <+\infty \}$ is dense in $\mathcal{A}$, where $\left[D, \pi(a)\right]$
 is the commutator operator. The operator $D$ is called the Dirac operator.
\end{enumerate}
\end{definition}

\begin{remark} \label{not}
    Note that we do not require that $D$ has compact resolvent.
\end{remark}

Here, for the spectral triple, we will take the Hilbert space \linebreak $\mathcal{H} = \lp{2}(\bm{\mu}) \times \lp{2}(\bm{\mu})$. 

We write $\mathcal{B} \defn \{\mathcal{L} \mid \lp{2}(\bm{\mu}) \to \lp{2}(\bm{\mu})\,|\, \mathcal{L} \text{ Linear bounded operator} \}$.

We consider the diagonal representation:
\begin{equation} \label{wer777}
    \pi: \mathcal{B} \to \{G =  \left(
\begin{array}{cc}
g_{11} &  g_{12}\\
g_{21}& g_{22}
\end{array}
\right)   : \lp{2}(\bm{\mu}) \times \lp{2}(\bm{\mu})\, \to  \lp{2}(\bm{\mu}) \times \lp{2}(\bm{\mu})\} \text,
\end{equation}
in such way that for $\mathcal{L}\in \mathcal{B}$ and $(\psi_1,\psi_2) \in \lp{2}(\bm{\mu}) \times \lp{2}(\bm{\mu})$ we get:
\begin{equation*}
    \pi(\mathcal{L}) (\psi_1,\psi_2)) = (\mathcal{L}(\psi_1),\mathcal{L}(\psi_2))=\left(
\begin{array}{cc}
\mathcal{L} &  0\\
0& \mathcal{L}
\end{array}
\right)         \left(
\begin{array}{c}
\psi_1\\
\psi_2
\end{array}
\right) \text.
\end{equation*}

For bounded operators, we consider the norm given by the spectral radius \(\rho (G)\).

$\mathcal{A} = \mathcal{B}$ is the $C^{*}$-algebra of bounded operators in $\lp{2}(\bm{\mu})$, $\pi$ is the diagonal representation
of \(\mathcal{A}\) in \(\mathcal{H} = \lp{2}(\bm{\mu}) \times \lp{2}(\bm{\mu})\) given by \eqref{wer777}.

$D = \dirac : \mathcal{H} \to \mathcal{H}$ is given by \eqref{plis} and we will show that the commutator $\left[\dirac, \pi(\mathcal{L})\right]$ can be
expressed by \eqref{wer77}.

We denote by $\mathcal{B}_{\mathfrak{H}}$ the set of Hermitian operators on $\lp{2}(\bm{\mu})$.

\begin{remark} \label{iint}
  A hermitian operator $\mathcal{L}$ in $\mathcal{B}$ is called an observable. In Quantum Mechanics the values obtained by
measuring the observable $\mathcal{L}$ is the set of (real) eigenvalues of $\mathcal{L}$.
\end{remark}

$\|\,\left[\mathcal{D}, \pi(\mathcal{L})\right]\,\|\leq 1$ corresponds to saying that the Lipschitz constant of the self-adjoint operator $\mathcal{L}\in \mathcal{B}_{\mathfrak{H}}$ is smaller than $1$.

Consider the $C^*$-algebra $\mathcal{B}$ of bounded operators acting on $\lp{2}(\bm{\mu})$, and denote by $\eta, \xi$ two general $C^*$-states of $\mathcal{B}$.

\begin{definition} \label{ccon}
    The Connes distance between the $C^*$-dynamical states \(\eta\) and \(\xi\) is:
    \begin{equation} \label{covoi}
      d_{\mathcal{D}}(\eta,\xi)  = \sup_{\substack{\mathcal{L} \in \mathcal{B} \\ \norm{\left[\mathcal{D},\pi(\mathcal{L})\right]} \leq 1}} \abs{ \eta(\mathcal{L}) - \xi(\mathcal{L}) } \text.
    \end{equation}
\end{definition}

\begin{remark} \label{Conm}
    The Connes distance corresponds to the $1$-Wasserstein distance among probabilities (see \cite{cfrconnes},
    \cite{Connes2}, \cite{KS}, \cite{KSS}, and \cite{CHLS}). The operator norm (spectral radius)
    $\|\, \left[\mathcal{D},\pi(\mathcal{L})\right] \,\|$ being less than or equal to one, should be {\it in some sense},
    analogous to saying that $\mathcal{L}$ has Lipschitz constant smaller than or equal to one. We will show that the value $\|\, \left[\mathcal{D},\pi(\mathcal{L})\right] \,\|$   will be given by expression  \eqref{aesq1}.
\end{remark}

For a given $\mathcal{L}$ in expression \eqref{covoi} it is quite important to be able to determine if either
$\norm{\left[\mathcal{D},\pi(\mathcal{L})\right]} \leq 1$ or $ \norm{\left[\mathcal{D},\pi(\mathcal{L})\right]} > 1$. 
One of the main issues here is to exhibit explicit examples where one can determine if either of the two options 
is true. This will help to provide lower bounds for $d_{\mathcal{D}}(\eta,\xi)$.

\begin{remark} \label{Conm27}
   On the issue of the estimation of  the Connes distance (via the sup in \eqref{covoi}), note that
   $\|\, \left[\mathcal{D},\pi(\mathcal{L})\right] \,\|$ is a seminorm\footnote{Absolutely homogeneous and subadditive.} on $\mathcal{L}$, and for practical purposes, a form of normalization on $\mathcal{L}$ is natural to be considered.
    It follows from \eqref{wer77}, \eqref{aesq1}, and triangle inequality that in the case $||\mathcal{L}||\leq 1/2$, we get that  $\|\, \left[\mathcal{D},\pi(\mathcal{L})\right] \,\|\leq1.$
\end{remark}

Below in Proposition \ref{Denk}, Theorems \ref{treyw} and \ref{corte}, and Example \ref{corte43},
we study the properties of $\left[\mathcal{D},\pi(\mathcal{L})\right]$ when $\mathcal{L}$ is a projection operator. As we are able to estimate
for several operators $\mathcal{L}$ (for instance the projection operators) when $\| \,\left[\mathcal{D},\pi(\mathcal{L})\right] \,\|\leq 1$,
we get indeed lower bounds for $d_{\mathcal{D}}(\eta, \xi)$.

For the general case $\mathcal{L}\in \mathcal{B}_{\mathfrak{H}}$ we need a more detailed analysis.

Given a continuous function $f$, we investigate a possible interpretation of the concept
$ \|\, \left[\mathcal{D}, \pi (\mult_f)\right] \,\|$, regarding either the associated forward discrete-time
dynamical derivative or the associated backward discrete-time dynamical derivative in Subsection \ref{pos}.

We will highlight now some of the main results obtained in our text.

One of our main purposes is to estimate  the expression:
$$\|\, \left[\mathcal{D} ,\pi(\koopman \ruelle)\right] \,\|\text{, and more generally } \|\, \left[\mathcal{D},\pi(\koopman^n \ruelle^n)\right] \,\| \text,$$
$n \in \mathbb{N}$. In Theorem \ref{EERT1} we show that $\|\, \left[\mathcal{D},\pi(\koopman^n \ruelle^n)\right] \,\|=1$.

Given a continuous function $f: \Omega \to  \mathbb{R}$, the multiplication operator $\mult_f$
is the one satisfying $g \to \mult_f(g)= g f$. The operators $\koopman^n \, \ruelle^n$
and $\mult_f$ are generators of the Exel-Lopes $C^*$-algebra introduced in \cite{EL1}
(see also \cite{EL2}). In Subsection \ref{pos} we also estimate $\|\, \left[\mathcal{D},\pi(\mult_f)\right] \,\|$.

In other direction, in Subsection \ref{pri}, we get some explicit results concerning $\|\, \left[\mathcal{D}, \pi(\mathcal{L})\right] \,\|$,
for operators $\mathcal{L}$ of different kind, like projection operators. For instance, taking into account the orthogonal
family of  H\"{o}lder  functions $e_w: \Omega \to \mathbb{R}$, indexed by finite words $w = w_1 w_2 \cdots w_l$,
$w_j \in \{0,1\}$, in Theorem \ref{agua} we get:
\begin{theorem} \label{agua12}
    Given a fixed word  $w = w_1 w_2 w_3 \cdots w_n$, $l(w)\geq2$, denote by $\hat{e}_w$ the projection
    operator on $e_w$. Then, we get for the operator norm:
    \begin{equation} \label{frio43}
        \|\, \left[\mathcal{D}, \pi(\hat{e}_w)\right] \,\| = 1 \text.
    \end{equation}
\end{theorem}

More generally, we get in Theorems \ref{treyw} and \ref{corte} in Subsection \ref{pri}:
\begin{theorem} \label{treu1}
    Given a non-constant $\lp{2}(\bm{\mu})$ function $\psi:\Omega \to \mathbb{R}$ such that $|\psi|=1$, denote
    by $\hat{\psi}$ the associated projection operator. Then, we get for the operator norm:
    \begin{align} \label{bodt1}
        \norm{\left[ \mathcal{D} , \pi(\hat{\psi}) \right]} & = \norm{(\koopman \hat{\psi} - \hat{\psi}  \koopman)} \nonumber \\
                                                            & = \sup_{\abs{\phi} = 1} \sqrt{\langle \phi, \psi \rangle^{2} - 2 \langle \phi, \psi \rangle \langle \koopman \phi, \psi \rangle \langle \koopman \psi, \psi \rangle + \langle \koopman \phi , \psi \rangle^{2}} \text.
    \end{align}
\end{theorem}

Expanding $\psi  = \sum_w b_w e_w + \beta_{0} e_{\varepsilon}^{0} + \beta_{1} e_{\varepsilon}^{1}$ via the
orthonormal family $e_w$, where $w$ ranges in the set of finite words, one can get an explicit expression
for $\norm{\left[ \mathcal{D} , \pi(\hat{\psi}) \right]}$ in terms of coefficients $b_w$, which will
be described by \eqref{fret}.

In another line of reasoning we will get Theorem \ref{Denk} in Subsection \ref{pri}:
\begin{theorem} \label{treu1bb}
    Given a non-constant $\lp{2}(\bm{\mu})$ function $\psi:\Omega \to \mathbb{R}$ such that $|\psi|=1$, denote
    by $\hat{\psi}$ the associated projection operator. Then, from the value $\langle \koopman(\psi),\psi\rangle$
    we will be able to get the explicit value $\norm{\left[ \mathcal{D} , \pi(\hat{\psi}) \right]}$ which satisfies:
    \begin{equation} \label{bodt1ba}
        \frac{3}{2\, \sqrt{2}} \geq \norm{\left[ \mathcal{D} , \pi(\hat{\psi}) \right]} \geq 1 \text.
    \end{equation}
\end{theorem}

Note that for a Hölder function $\psi$ we get  $\langle \koopman(\psi),\psi\rangle=1$ only when $\psi$ is constant.

Following a different rationale in Proposition \ref{corte} we get:

\begin{proposition} \label{corte12aa}
    Suppose $\psi$ with norm $1$ is a Hölder function in the kernel of the Ruelle operator $\ruelle$, then:
    \begin{equation} \label{ljg}
        \norm{\left[ \mathcal{D} , \pi(\hat{\psi}) \right]}=1 \text.
    \end{equation}
    If $\psi$ is of the form $\psi = \koopman^k (f)$, $k\geq 1$, where $f$ is in the kernel of the
    Ruelle operator $\ruelle$, we get the same equality.
\end{proposition}

Moreover, from Proposition \ref{corte43}
\begin{proposition} \label{corte43a} For a {\it generic} Hölder function $\psi$ with norm $1$
    \begin{equation}
        \norm{\left[ \mathcal{D} , \pi(\hat{\psi}) \right]}>1 \text.
    \end{equation}
    We get \eqref{ljg} just when  $\langle \koopman(\psi),\psi\rangle=0$.
\end{proposition}

We address the issue of possible interpretation of $\norm{\left[ \mathcal{D} , \pi ( \mult_{f} ) \right]}$
regarding forms of discrete-time dynamical derivatives for $f$. In Subsection \ref{pos} we show for
the multiplication operator $\mult_f$ the following result:
\begin{theorem} \label{trocad112}
    For any $f \in C(\Omega)$:
    \begin{equation} \label{lkor23}
        |f - f \circ \sigma|_{\infty} = \abs{\koopman f - f}_{\infty}  \geq \norm{\left[ \mathcal{D} , \pi ( \mult_{f} ) \right]} \geq \abs{f - \ruelle f}_{\infty} \text.
    \end{equation}
The left-hand side of \eqref{lkor23} concerns a form of the supremum of {\it dynamical mean forward derivative}
for $f$. We will get equality on both sides when $f$ does not depend on the first coordinate.
\end{theorem}

Moreover,
\begin{proposition} \label{erte2766}
    For any \(f \in C(\Omega)\):
    \begin{align} \label{drms34}
        \norm{\left[ \mathcal{D} , \pi ( \mult_{f} ) \right]} & = \abs{\sqrt{\ruelle \abs{\koopman f - f}^{2}}}_{\infty} \nonumber \\
                                                                    & = \sup_{x \in \Omega} \sqrt{\frac{\abs{f(x) - f(0x)}^{2}}{2} + \frac{\abs{f(x) - f(1x)}^{2}}{2}} \text.
    \end{align}
    The right-hand side of \eqref{drms34} is a special form of supremum of {\it mean backward derivative} for $f$.
\end{proposition}
% We highlight the fact that obtaining operators $\mathcal{L}$   satisfying $\norm{\left[ D , \pi(\mathcal{L}) \right]}=1$(as we get from  Theorem \ref{agua12} and Corollary \ref{corte12}) helps  to derive lower bounds for the Connes distance between density operators acting  in the $\lp{2}(\bm{\mu})$ space (see Definition \ref{ccon}).

A related work which considers the Dirac Operator and spectral triples for the context of $L^p$-algebras appears in \cite{rk-lp}.

In \cite{LCH} the authors compute explicitly the Connes distance in a finite-dimensional case for the
2D fermionic space (a non-dynamical setting).

Results for spectral triples in a dynamical context can be found in \cite{CHLS}, \cite{JuPu}, \cite{KS},
\cite{KSS}, \cite{Sha1}, and  \cite{Sha2}.

\section{The Dirac operator} \label{didi}

For normal (and in particular, both self-adjoint and anti-self-adjoint) elements, the norm considered
here becomes:
\begin{equation*}
    \norm{G} = \sup_{\substack{(\phi, \psi) \in \lp{2}(\d \bm{\mu}) \times \lp{2}(\d \bm{\mu}) \\ \abs{\phi}^{2} + \abs{\psi}^{2} = 1}} \abs{G(\phi, \psi)} \text.
\end{equation*}

Furthermore, notice that concerning operators of the form:
\begin{align*}
    T = \left( \begin{array}{cc} 0 &  A_{1} \\ A_{2} & 0 \end{array} \right) && \text{or of the form:} && T^{'} = \left( \begin{array}{cc} A_{1}^{'} &  0 \\ 0 & A_{2}^{'} \end{array} \right) \text,
\end{align*}
the operator norm (not necessarily equal to the spectral radius norm) becomes:
\begin{align*}
    \norm{T} & = \sup_{\substack{(\phi, \psi) \in \lp{2}(\d \bm{\mu}) \times \lp{2}(\d \bm{\mu}) \\ \abs{\phi}^{2} + \abs{\psi}^{2} = 1}} \abs{T(\phi, \psi)} \\
             & = \sqrt{\sup_{\substack{(\phi, \psi) \in \lp{2}(\d \bm{\mu}) \times \lp{2}(\d \bm{\mu}) \\ \abs{\phi}^{2} + \abs{\psi}^{2} = 1}} \abs{T(\phi, \psi)}^{2}} \\
             & = \sqrt{\sup_{\substack{(\phi, \psi) \in \lp{2}(\d \bm{\mu}) \times \lp{2}(\d \bm{\mu}) \\ \abs{\phi}^{2} + \abs{\psi}^{2} = 1}} \abs{A_{1} \psi}^{2} + \abs{A_{2} \phi}^{2}} \\
             & \leq \sqrt{\sup_{\substack{(\phi, \psi) \in \lp{2}(\d \bm{\mu}) \times \lp{2}(\d \bm{\mu}) \\ \abs{\phi}^{2} + \abs{\psi}^{2} = 1}} \norm{A_{1}}^{2} \abs{\psi}^{2} + \norm{A_{2}}^{2} \abs{\phi}^{2}} \\
             & \leq \sqrt{\sup_{\substack{(\phi, \psi) \in \lp{2}(\d \bm{\mu}) \times \lp{2}(\d \bm{\mu}) \\ \abs{\phi}^{2} + \abs{\psi}^{2} = 1}} \max \left\{ \norm{A_{1}}^{2}, \norm{A_{2}}^{2} \right\} \abs{\psi}^{2} + \abs{\phi}^{2}} \\
             & = \max \left\{ \norm{A_{1}} , \norm{A_{2}} \right\} \\
             \intertext{or}
    \norm{T^{'}} & = \sup_{\substack{(\phi, \psi) \in \lp{2}(\d \bm{\mu}) \times \lp{2}(\d \bm{\mu}) \\ \abs{\phi}^{2} + \abs{\psi}^{2} = 1}} \abs{T^{'}(\phi, \psi)} \\
             & = \sqrt{\sup_{\substack{(\phi, \psi) \in \lp{2}(\d \bm{\mu}) \times \lp{2}(\d \bm{\mu}) \\ \abs{\phi}^{2} + \abs{\psi}^{2} = 1}} \abs{T^{'}(\phi, \psi)}^{2}} \\
             & = \sqrt{\sup_{\substack{(\phi, \psi) \in \lp{2}(\d \bm{\mu}) \times \lp{2}(\d \bm{\mu}) \\ \abs{\phi}^{2} + \abs{\psi}^{2} = 1}} \abs{A_{1}^{'} \phi}^{2} + \abs{A_{2}^{'} \psi}^{2}} \\
             & \leq \sqrt{\sup_{\substack{(\phi, \psi) \in \lp{2}(\d \bm{\mu}) \times \lp{2}(\d \bm{\mu}) \\ \abs{\phi}^{2} + \abs{\psi}^{2} = 1}} \norm{A_{1}^{'}}^{2} \abs{\phi}^{2} + \norm{A_{2}^{'}}^{2} \abs{\psi}^{2}} \\
             & \leq \sqrt{\sup_{\substack{(\phi, \psi) \in \lp{2}(\d \bm{\mu}) \times \lp{2}(\d \bm{\mu}) \\ \abs{\phi}^{2} + \abs{\psi}^{2} = 1}} \max \left\{ \norm{A_{1}^{'}}^{2}, \norm{A_{2}^{'}}^{2} \right\} \abs{\phi}^{2} + \abs{\psi}^{2}} \\
             & = \max \left\{ \norm{A_{1}^{'}} , \norm{A_{2}^{'}} \right\} \text.
\end{align*}
Finally, one may use a simple approximation argument to show the opposite inequality. We will do it
here for an operator of the first form,  assuming, without loss of generality, that
\(\norm{A_{1}} \geq \norm{A_{2}}\).

Let \((\psi_{n})_{n}\) be a sequence of functions such that \(\abs{\psi_{n}} = 1\), and
\(\lim_{n} \abs{A_{1} \psi_{n}} = \norm{A_{1}}\). Then \(\left((0, \psi_{n})\right)_{n}\) is a sequence of
pairs of functions such that \(\abs{(0, \psi_{n})} = 1\) and:
\begin{align*}
    \lim_{n} \abs{T (0, \psi_{n})} & = \lim_{n} \abs{(A_{1} \psi_{n} , 0)} \\
                                   & = \lim_{n} \abs{A_{1} \psi_{n}} \\
                                   & = \norm{A_{1}} \text.
\end{align*}
This implies \(\norm{T} \geq \norm{A_{1}}\) and consequently \(\norm{T} = \norm{A_{1}}\).

Thus:
\begin{equation}  \label{kju}
    \norm{T} \,\, = \max \left\{ \norm{A_{1}} , \norm{A_{2}} \right\}, \,\, \text{ and:} \,\, \norm{T^{'}} \,\, = \max \left\{ \norm{A_{1}^{'}} , \norm{A_{2}^{'}} \right\}.
\end{equation}

Notice our Dirac operator, which is self-adjoint, is then bounded, with a norm equal to the maximum
of the norms of \(\koopman\) and \(\ruelle\). These two operators being adjoint to one another
have the same norm, and thus, \(\norm{\mathcal{D}} = \norm{\koopman} = \norm{\ruelle} = 1\).
But if \(\dirac\) has a spectral radius equal to \(1\), its spectrum is bounded, and thus it cannot have
compact resolvent.

We get that for each $\mathcal{L}$
\begin{align} \label{wer77}
    \left[\mathcal{D}, \pi(\mathcal{L})\right] & = \sqrt{2} \left(
\begin{array}{cc}
0 & \hat{f}^\dag \mathcal{L} - \mathcal{L} \hat{f}^\dag \\
\hat{f} \mathcal{L} - \mathcal{L} \hat{f} & 0
\end{array}
\right) \nonumber \\
                          & = \left(
\begin{array}{cc}
0 & \koopman \mathcal{L} - \mathcal{L} \koopman \\
\ruelle \mathcal{L} - \mathcal{L} \ruelle & 0
\end{array}
\right) \text.
\end{align}

If $\mathcal{L}$ is self-adjoint, it will follow from \eqref{kju}, \eqref{wer77}, and Lemma \ref{aesq2}, that:
\begin{equation} \label{aesq1}
\norm{\left[\mathcal{D}, \pi(\mathcal{L})\right]}=  \norm{(\ruelle \mathcal{L} - \mathcal{L} \ruelle)} = \norm{(\koopman \mathcal{L} - \mathcal{L}  \koopman)}.
\end{equation}

Proposition \ref{agua} will exhibit non-trivial operators $\mathcal{L}$ such that $\norm{\left[\mathcal{D}, \pi(\mathcal{L})\right]}=1$.

We will consider later for each word $w$ the projection $\hat{e}_w$ on $e_w$:
$$ \psi \to \hat{e}_w (\psi)= \langle \psi, e_w \rangle\, e_w.$$

In this way,
\begin{equation} \label{wer778}
    \left[\mathcal{D}, \pi(\hat{e}_w)\right] = \left(
\begin{array}{cc}
0 &   \koopman \hat{e}_w - \hat{e}_w  \koopman\\
\ruelle \hat{e}_w  - \hat{e}_w   \ruelle & 0
\end{array}
\right) \text.
\end{equation}

Our purpose from now on is to estimate the value of $\|\, \left[\mathcal{D}, \pi(\hat{e}_w)\right] \,\|$.

\begin{lemma} \label{aesq2}
    If $\mathcal{L}$ is self-adjoint, then:
    $$ \norm{(\ruelle \mathcal{L} - \mathcal{L} \ruelle)} = \norm{(\koopman \mathcal{L} - \mathcal{L}  \koopman)} \text. $$

    In particular, if  $\hat{\psi}$ is the projection on the unitary vector $\psi\in \lp{2} (\bm{\mu})$, then:
    $$ \norm{(\ruelle \hat{\psi} - \hat{\psi}  \ruelle)} = \norm{(\koopman \hat{\psi} - \hat{\psi}  \koopman)} \text. $$

    The main conclusion is:
    \begin{align} \label{pore1}
        \norm{\left[ \dirac , \pi(\hat{\psi}) \right]} & = \max{\left\{\norm{(\koopman \hat{\psi} - \hat{\psi}  \koopman)}, \norm{(\ruelle \hat{\psi} - \hat{\psi}  \ruelle)}\right\}} \nonumber \\
                                                  & = \norm{(\koopman \hat{\psi} - \hat{\psi}  \koopman)} \text.
    \end{align}
\end{lemma}

\begin{proof}
Notice that:
\begin{align*}
    \norm{A}^{2} & = \sup_{\abs{\phi} = 1} \abs{A \phi}^{2} \\
             & = \sup_{\abs{\phi} = 1} \langle A \phi, A \phi \rangle \\
             & = \sup_{\abs{\phi} = 1} \langle A^{\dagger} A \phi, \phi \rangle \\
             & \leq \sup_{\abs{\phi} = 1} \abs{A^{\dagger} A \phi} \abs{ \phi } \\
             & \leq \sup_{\abs{\phi} = 1} \norm{A^{\dagger} A}\abs{ \phi}^{2} \\
             & \leq \norm{A^{\dagger}}\norm{A} \text,
\end{align*}
Which implies:
\begin{align*}
    \norm{A} \leq \norm{A^{\dagger}} \text.
\end{align*}
If we take \(A = A^{\dagger}\), this means:
\begin{align*}
    \norm{A^{\dagger}} \leq \norm{A^{\dagger\dagger}} = \norm{A} \text,
\end{align*}
And thus:
\begin{align*}
    \norm{A} = \norm{A^{\dagger}} \text.
\end{align*}

Then, if \(T\) is a self-adjoint operator, it follows that:
\begin{align*}
    \norm{\ruelle T - T \ruelle} & = \norm{\left( \ruelle T - T \ruelle \right)^{\dagger}} \\
                                         & = \norm{T \ruelle^{\dagger} - \ruelle^{\dagger} T} \\
                                         & = \norm{T \koopman - \koopman T} \\
                                         & = \norm{- \left(T \koopman - \koopman T\right)} \\
                                         & = \norm{\koopman T - T \koopman} \text.
\end{align*}
By taking \(T = \hat{\psi}\), a projection on the element $\psi$ of norm $1$ we get:
\begin{align*}
    \norm{\ruelle \hat{\psi} - \hat{\psi} \ruelle} = \norm{\koopman \hat{\psi} - \hat{\psi} \koopman} \text.
\end{align*}
\end{proof}

\subsection{Estimates in the case of projection operators} \label{pri}

In this subsection, we consider operators $\mathcal{L}$ such that $\mathcal{L}=\hat{\psi}$ is a projection on a Hölder
element $\psi$ with $\lp{2}$ norm equal to $1$. We want to estimate:
\begin{equation}
    \norm{\left[ \dirac , \pi(\hat{\psi}) \right]} \text.
\end{equation}

\begin{proposition} \label{lal1}
  Suppose $e_w =e_{w_1 w_2 \cdots w_n}$, $l(w)>1$, then:
  \begin{align*}
      (\koopman \hat{e}_w - \hat{e}_w  \koopman) (\phi) & =
\sqrt{\frac{1}{2}}  [ \, (\int e_w\, \phi \d \bm{\mu} )  ( e_{0w} + e_{1w} ) -  \int e_{w_2 w_3 \cdots w_n}\,\phi \d \bm{\mu} \,\,\,e_w \,\,] \\
                                                              & = \frac{1}{\sqrt{2}} \left[ \langle e_{w} , \phi \rangle \left( e_{0w} + e_{1w} \right) - \langle e_{\sigma(w)} , \phi \rangle e_{w} \right] \text, \\
                                                              \intertext{and}
      (\ruelle \hat{e}_w - \hat{e}_w  \ruelle) (\phi) & =
\sqrt{\frac{1}{2}} \,[ \int e_w\,\phi \d \bm{\mu}\,\,  e_{w_2 \cdots w_n}   - \, \int ( e_{0w} + e_{1w} )\,\phi \d \bm{\mu}\,  \, \, \,e_w\,] \\
                                                              & = \frac{1}{\sqrt{2}} \left[ \langle e_{w} , \phi \rangle e_{\sigma(w)} - \langle e_{0w} + e_{1w} , \phi \rangle e_{w} \right]\text. \\
                                                              \intertext{Moreover,}
      (\koopman \hat{e}_{e_{\varepsilon}^0} - \hat{e}_{e_{\varepsilon}^0}  \koopman) (\phi) & =-\sqrt{2} \,\langle e_{\varepsilon}^0 , \phi \rangle ( \chi_{00} + \chi_{10}) + \langle e_{\varepsilon}^0 , (\phi \circ \sigma) \rangle e_{\varepsilon}^0 \text, \\
      (\koopman \hat{e}_{e_{\varepsilon}^1} - \hat{e}_{e_{\varepsilon}^1}  \koopman) (\phi) & =\sqrt{2} \,\langle e_{\varepsilon}^1 , \phi \rangle ( \chi_{01} + \chi_{11})- \langle e_{\varepsilon}^1 , (\phi \circ \sigma) \rangle e_{\varepsilon}^1 \text, \\
      (\ruelle \hat{e}_{e_{\varepsilon}^0} - \hat{e}_{e_{\varepsilon}^0}  \ruelle) (\phi) & =-\sqrt{\frac{1}{2}} \,\langle e_{\varepsilon}^0 , \phi \rangle + \sqrt{2} \langle ( \chi_{00} + \chi_{10})  ,\phi \rangle \,e_{\varepsilon}^0 \text, \\
             \intertext{and}
      (\ruelle \hat{e}_{e_{\varepsilon}^1} - \hat{e}_{e_{\varepsilon}^1}  \ruelle) (\phi) & = \sqrt{\frac{1}{2}}\, \langle e_{\varepsilon}^1 , \phi \rangle - \sqrt{2} \langle ( \chi_{01} + \chi_{11})  ,\phi \rangle \,e_{\varepsilon}^1 \text.
  \end{align*}
\end{proposition}
For the proof see Proposition \ref{lal1a}.

%  Given a  generic word $\tilde{w}=[w_1,w_2,...,w_k]<w= [w_1,w_2,...,w_k,w_{k+1},...,w_n],$
% Summarizing what was described above:
\begin{proposition} \label{okl}
    For a word $w$ satisfying $l(w)>1$, given a  generic word $\tilde{w}$, and the corresponding element $e_{\tilde{w}}$, we get that:
    \begin{equation} \label{frio1}
        |(\koopman \hat{e}_w - \hat{e}_w  \koopman) (e_{\tilde{w}})|^2 = \int [\, (\koopman \hat{e}_w - \hat{e}_w  \koopman) (e_{\tilde{w}} )\,]^2 \d \bm{\mu}
    \end{equation}
    is equal to $0$, if  $w\neq \tilde{w}\neq w_2 w_3 \cdots w_n$, is equal to $1/2$ if $w\neq\tilde{w}= w_2 w_3 \cdots w_n$, and is equal to 1 if $\tilde{w}=w$.  Moreover,
    \begin{equation}
        |(\ruelle \hat{e}_w - \hat{e}_w  \ruelle) (e_{\tilde{w}})|^2 = \int [\, (\ruelle \hat{e}_w - \hat{e}_w  \ruelle) (e_{\tilde{w}} )\,]^2 \d \bm{\mu}
    \end{equation}
    is equal to $0$, if  $\tilde{w} \neq w \neq \sigma{\tilde{w}}$, and is equal to $1/2$ if either $\tilde{w}\neq w = \sigma{\tilde{w}}$ or $\tilde{w}=w$.
\end{proposition}
For the proof see Proposition \ref{okla} in the Appendix.

Using the above results we get:

\begin{theorem} \label{agua}
    Given a fixed word  $w= w_1 w_2 w_3 \cdots w_n$, $l(w)\geq2$, we get for the operator norm:
    \begin{align} \label{frio3}
        \|\, (\koopman \hat{e}_w - \hat{e}_w  \koopman) \,\| = 1 \text, \\
    \intertext{and also:}
        \norm{(\ruelle \hat{e}_w - \hat{e}_w  \ruelle)} = 1 \text; \\
    \intertext{Therefore:} \label{frio4}
        \norm{\left[\mathcal{D}, \pi(\hat{e}_w)\right]} = 1 \text.
    \end{align}
\end{theorem}
For proof see Proposition \ref{aguaa} in the  Appendix.

We denote by $\mathcal{B}_{\operatorname{Lip}}$ the set:
\begin{equation} \label{wer1}
    \mathcal{B}_{\operatorname{Lip}} \defn \{\mathcal{L} \in \mathcal{A}\, \mid \norm{\left[\mathcal{D}, \pi(\mathcal{L})\right]} \, \leq 1\} \text.
\end{equation}

%Denote by $\left\{\beta_n\right\}$ any Hilbert-basis for $\lp{2}$.
%The trace of an hermitian operator $V:  \lp{2} (\bm{\mu}) \to \lp{2}(\bm{\mu}) $ is $\trace (V)=\int \sum_n V (\beta_n) \beta_n \d \bm{\mu}.$

%An operator $P:\lp{2}(\bm{\mu}) \to \lp{2}(\bm{\mu})$ is called a density operator if $P$ is hermitian with trace $1$
%and all of its eigenvalues are non-negative.

%The expected value of an hermitian operator $A : \lp{2} (\bm{\mu}) \to \lp{2}(\bm{\mu}) $ for the action of the density
%operator $P: \lp{2} (\bm{\mu}) \to \lp{2}(\bm{\mu}) $ is $\mathbb{E}_{P}(A) = \trace (P A)$
%
%Given an element $\phi$ is $\lp{2}(\bm{\mu})$ such that $|\phi|=1$ and $\hat{\phi}$ the projection on $\phi$,
%we denote $\mathbb{E}_P(\phi)= \trace (P \hat{\phi})$.

Given $\psi$ such that $|\psi|=1$, we consider the projection $\hat{\psi}$:
$$ \phi \to \hat{\psi} (\phi)= \langle \phi,\psi \rangle\, \psi \text.$$

Assume that:
\begin{equation} \label{gel1}
    \psi = \sum_u b_u e_u + \beta_{0} e_{\varepsilon}^{0} + \beta_{1} e_{\varepsilon}^{1} \text,
\end{equation}
in such way that $\sum_u b_u^2 + \beta_{0}^{2} + \beta_{1}^{2} = 1$, and:
\begin{equation} \label{gel2}
    \phi = \sum_u a_u e_u + \alpha_{0} e_{\varepsilon}^{0} + \alpha_{1} e_{\varepsilon}^{1} \text,
\end{equation}
in such way that $\sum_u a_u^2 + \alpha_{0}^{2} + \alpha_{1}^{2} = 1$.

%Note that
%$$<\phi,\psi>\,\psi=(\sum_u a_u b_u + \alpha_{0} \beta_{0} + \alpha_{1} \beta_{1})\, \psi.= $$
%$$ (\sum_u a_u b_u + \alpha_{0} \beta_{0} + \alpha_{1} \beta_{1}) \, \sum_u b_u e_u + \beta_{0} e_{\varepsilon}^{0} + \beta_{1} e_{\varepsilon}^{1} $$

One can show that:
\begin{equation} \label{lili1}
    \koopman(\psi)  = \sum_u b_u \frac{1}{\sqrt{2}}  ( e_{0u} + e_{1u} ) - \sqrt{2}\beta_0( \chi_{00} + \chi_{10} )+  \sqrt{2}\beta_1( \chi_{01} + \chi_{11} ) \text.
\end{equation}

%We denote $A_\psi:= \koopman(\psi) .$

Note also that:
\begin{align*}
    \ruelle (\phi) & = \sum_{l(v) \geq 2} a_{v} \frac{1}{\sqrt{2}} e_{\sigma(v)} + \ruelle \left( a_{0} e_{0} + a_{1} e_{1} + \alpha_{0} e_{\varepsilon}^{0} + \alpha_{1} e_{\varepsilon}^{1}\right)\\
                     & = \sum_{l(v) \geq 2} a_{v} \frac{1}{\sqrt{2}} e_{\sigma(v)} + \frac{1}{2} \left( a_{0} + a_{1} \right) \left(e_{\varepsilon}^{0} + e_{\varepsilon}^{1}\right) + \frac{1}{\sqrt{2}} \left(\alpha_{1} - \alpha_{0}\right)\\
                     & = \sum_{l(v) \geq 2} a_{v} \frac{1}{\sqrt{2}} e_{\sigma(v)} + \frac{1}{2} \left( a_{0} + a_{1} \right) \left(e_{\varepsilon}^{0} + e_{\varepsilon}^{1}\right) + \frac{1}{\sqrt{2}} \left(\alpha_{1} - \alpha_{0}\right)\left(e_{\varepsilon}^{1} - e_{\varepsilon}^{0}\right)\\
                     & = \sum_{l(v) \geq 2} a_{v} \frac{1}{\sqrt{2}} e_{\sigma(v)} + \frac{1}{\sqrt{2}} \left( \frac{1}{\sqrt{2}} \left( a_{0} + a_{1} \right) - \left(\alpha_{1} - \alpha_{0}\right) \right) e_{\varepsilon}^{0} + \\
                    & \quad \quad + \frac{1}{\sqrt{2}} \left( \frac{1}{\sqrt{2}} \left( a_{0} + a_{1} \right) + \left(\alpha_{1} - \alpha_{0}\right) \right) e_{\varepsilon}^{1} \text.
\end{align*}

Denote:
\begin{align} \label{lili}
    c(\psi) & \defn \langle \koopman(\psi) , \psi \rangle \nonumber \\
           & = \sum_{u}\frac{b_u}{\sqrt{2}}  (b_{0u}+  b_{1u}) + \frac{1}{2} (b_0 + b_1) (\beta_0 + \beta_1) \text.
\end{align}

Note that $|\koopman(\psi)|=1$ and $|c(\psi)| \leq 1$.

Given $\phi,\psi$, an estimation of the value:
\begin{equation} \label{mestg}
    \langle \phi, \psi \rangle^{2} - 2 \langle \phi, \psi \rangle \langle \koopman \phi, \psi \rangle \langle \koopman \psi, \psi \rangle + \langle \koopman \phi , \psi \rangle^{2}
\end{equation}
will be important regarding the future Theorem \ref{treyw} (see expression \eqref{req}).
We will address this issue soon and this requires to estimate $c(\psi)$.

\begin{example} \label{esqe}
    As an example, take a fixed  finite word $w$ on the symbols $\left\{0, 1\right\}$, with $l(w)>1$,
    and $\psi$ of the form:
    \begin{equation} \label{mmuc1}
        \psi= b_w e_w + b_{0w}  e_{0w} +  b_{1w}  e_{1w} \text,
    \end{equation}
    where $b_w^2+ b_{0w}^2   +  b_{1w}^2=1$. Then,
    \begin{equation} \label{mmuc2}
    \koopman(\psi) =\frac{1}{\sqrt{2}} [     (  e_{0w} +  e_{1w}  )  b_{w} +  (  e_{00w} +  e_{10w}  )  b_{0w}+  (  e_{01w} +  e_{11w}  )  b_{1w}] \text,
    \end{equation}
    and:
    \begin{equation} \label{mmuc3}
        c(\psi) = \langle \koopman(\psi) , \psi \rangle = \frac{b_w}{\sqrt{2}}  (b_{0w}+  b_{1w}) \text.
    \end{equation}

    Considering this kind  of $\psi$, for the future expression  \eqref{req}  we get:
    $$\sqrt{ b_{w}^{2} + \frac{1}{2} \left( b_{0w} + b_{1w} \right)^{2} - 2 b_{w} \frac{1}{\sqrt{2}} \left( b_{0w} + b_{1w} \right) c(\psi)}= $$
    $$\sqrt{ b_{w}^{2} + \frac{1}{2} \left( b_{0w} + b_{1w} \right)^{2} -  b_{w}^2  \left( b_{0w} + b_{1w} \right)^2} \text,$$
    which can be shown to be smaller than or equal to $1$, and to be equal to $1$ only when $b_w=1,  b_{0w} =0= b_{1w}$.
    In this case, it follows that $\psi$ is of the form $\psi=e_w$.
\end{example}

We thank M. Denker for a suggestion we used on the proof of the next result.

\begin{theorem} \label{Denk}
    Consider a real Hilbert space $\mathcal{H}$ and the action of a linear operator $\mathcal{W}:\mathcal{H} \to \mathcal{H}$
    that preserves the inner product. Then for a fixed $\psi \in \mathcal{H}$ with norm $1$, and any
    variable $\phi$ with norm $1$, we get that:
    \begin{equation} \label{estg}
        \frac{9}{8} \geq \,\langle \phi, \psi \rangle^{2} - 2 \langle \phi, \psi \rangle \langle \mathcal{W} \phi, \psi \rangle \langle \mathcal{W} \psi, \psi \rangle + \langle \mathcal{W} \phi , \psi \rangle^{2}\geq 1 \text.
    \end{equation}

    There exists an element $\psi$, and a particular $\phi$ in the two dimensional linear space
    $\{\psi , \mathcal{W}^{\dagger} (\psi)\}$, where the maximal value $9/8$ is attained.

    The upper bound in claim \eqref{estg} is true for either the Koopman operator $\koopman$, or a unitary operator.

    In order to get the value $1$ in \eqref{estg} we have to take $\langle\mathcal{W}(\psi) ,\psi\rangle=1$
    or $\langle \mathcal{W}(\psi) ,\psi\rangle=0$. For the Koopman operator $\langle\koopman(\psi) ,\psi\rangle=1$,
    only when $\psi=1$.
\end{theorem}
\begin{proof}
  Given $\psi$ of norm $1$, consider the linear subspace space $Y$ generated by the basis $\{\psi , \mathcal{W}^{\dagger} (\psi)\}$.

  Denote $c= \cos (\theta)=\langle \psi , \mathcal{W}^{\dagger} (\psi)\rangle =  \langle  \mathcal{W}(\psi) ,\psi\rangle \in [-1,1]$.
    Note that $\mathcal{W}(\psi)$ does not necessarily belong to $Y$.

    For fixed $\psi$ with norm $1$, given a vector $\phi \in \mathcal{H}$ of norm $1$, denote
    $\phi= \alpha \phi_1 + \beta \phi_2$, where $|\psi_1|=1=|\psi_2|$, $\phi_1\in Y$ and $\phi_2$ is
    orthogonal to the linear subspace $Y$. In this case $\alpha^2 + \beta^2=1$,  $\langle\phi_2, \psi\rangle=0$
    and $\langle\phi_2, \mathcal{W}^{\dagger}(\psi)\rangle=0$.

    Denote $\phi_1= a \,\psi + b\, \hat{\psi}$, where $|\psi|=1=|\hat{\psi}|$, $\langle\psi,\hat{\psi}\rangle=0$,
    $\hat{\psi}\in Y$, $a,b\neq0$ $a^2 + b^2=1$. Also, denote
    $d=\langle \hat{\psi}, \mathcal{W}^{\dagger}  ( \psi)\rangle= \langle \mathcal{W} (  \hat{\psi}),\psi\rangle= \cos(\theta \pm \frac{\pi}{2})=\pm \sin(\theta) \in [-1,1]$.
    Note that $\mathcal{W}(\hat{\psi})$ does not necessarily belongs to $Y$.

    Then,
    $$\,\langle \phi, \psi \rangle^{2} - 2 \langle \phi, \psi \rangle \langle \mathcal{W} \phi, \psi \rangle \langle \mathcal{W} \psi, \psi \rangle +    \langle \mathcal{W} \phi , \psi \rangle^{2}\,=$$
    $$\langle \alpha \phi_1 + \beta \phi_2, \psi \rangle^2 - \langle \alpha \phi_1 +\beta \phi_2, \psi \rangle   \langle \alpha  \mathcal{W}  (\phi_1)  + \beta \mathcal{W} (\phi_2), \psi \rangle\, c + $$
    $$  \langle \alpha  \mathcal{W}  (\phi_1) + \beta \mathcal{W} (\phi_2), \psi \rangle^2=$$
    $$ \langle \alpha \phi_1 , \psi \rangle^2  - \langle \alpha \phi_1 , \psi \rangle   \langle \alpha  \mathcal{W}  (\phi_1) ,\psi\rangle\, c + \langle \alpha  \mathcal{W}  (\phi_1), \psi \rangle^2 = $$
    $$ \langle \alpha ( a \,\psi + b\, \hat{\psi}), \psi \rangle^2  - \langle \alpha  ( a \,\psi + b\,  \hat{\psi}) , \psi \rangle   \langle \alpha (  a \mathcal{W} (\psi) + b   \mathcal{W} (  \hat{\psi})), \psi \rangle\,  c +$$
    $$ \langle \alpha (  a \mathcal{W} (\psi) + b   \mathcal{W} ( \hat{\psi})),\psi) \rangle^2 = $$
    $$ \alpha^2  a^2 - \alpha  a \,\alpha (a c + b \langle \mathcal{W} (  \hat{\psi})), \psi \rangle  )\, c + \alpha ^2 (a c + b \langle \mathcal{W} (  \hat{\psi})), \psi \rangle  )^2 =$$
    \begin{equation} \label{mmuc5}
        \alpha^2 \,[    a^2 -   a  \, (a c + b d  )\, c +  (a c + b d  )^2] \text,
    \end{equation}
    where:
    \begin{equation} \label{nene}
        |\alpha|<1, a^2 + b^2=1=c^2 + d^2 \text.
    \end{equation}

    Given a fixed $\alpha$, the maximal value of $    a^2 -   a  \, (a c + b d  )\, c +  (a c + b d  )^2$,
    under the constraints \eqref{nene}, is the value $9/8$ attained when:
    $$a =-\frac{\sqrt{ 3}}{2}  ,c=  -  \frac{1}{2}, \,\text{or}\,\,c= \frac{1}{2} \text.$$

    When $\phi=\phi_1\in Y$, we get get $\alpha=1$, which maximizes \eqref{mmuc5}. Taking $\alpha=1$
    means $a=\langle\phi,\psi\rangle$.

    For the case of the Koopman operator $\koopman$ (associated with the maximal entropy measure
    on $\{0,1\}^\mathbb{N}$) the maximal value can be realized. Indeed, given any finite word $w$,
    consider in  Example \ref{esqe} the values $b_w= \frac{1}{\sqrt2} , b_{0w} =-1/2 =b_{1w}$. Then, we get  $\langle \koopman  \psi,\psi\rangle=-\frac{1}{2}=c,$ when $\psi$ is of the form $\psi= b_w e_w + b_{0w}  e_{0w} +  b_{1w}  e_{1w}$. For such $\psi$,
    the function $\hat{\psi}=\frac{1}{\sqrt2} e_w + \frac{1}{2}  e_{0w} +  \frac{1}{2}  e_{1w}$ is orthogonal to $\psi$ and has norm $1$.  Take $\phi=\phi_1\in Y$, and moreover set:
    $$ \phi_1=  -\frac{\sqrt{ 3}}{2} \psi + \frac{1}{2}\hat{\psi}.$$
    In this case, for such $\psi$, taking such $\psi_1$ we get the maximal value:
    $$\,\langle \phi, \psi \rangle^{2} - 2 \langle \phi, \psi \rangle \langle \mathcal{W} \phi, \psi \rangle \langle \mathcal{W} \psi, \psi \rangle +    \langle \mathcal{W} \phi , \psi \rangle^{2}\,=9/8.$$

    Consider the function $G(a,c) =   a^2 -   a  \, (a c + b d  )\, c +  (a c + b d  )^2\, $.

    For a fixed value $c$ the partial derivative:
    \begin{align} \label{tororo}
        \frac{G(a,c)}{\partial a} & = \frac{\partial \,[\,  a^2 -   a  \, (a c + b d  )\, c +  (a c + b d  )^2\,] }{\partial a} \nonumber \\
                                  & = \frac{ c \,\,\,(2 \,a\,c\, \sqrt{1 - a^2}\, + \sqrt{1 - c^2} - 2\, a^2 \,\sqrt{1 - c^2})}{\sqrt{1 - a^2}} \text.
    \end{align}

    Given $c$ the value $a=a_c$ such that $2 \,a\,c\, \sqrt{1 - a^2}\, + \sqrt{1 - c^2} - 2\, a^2 \,\sqrt{1 - c^2}=0$,
    producing the largest value $G(a_c,c)$ is:
    $$a_c = \frac{\sqrt{1 + c}}{\sqrt{2}} \text,$$
    and the largest value is:
    $$G(a_c,c)= \frac{1}{2} (2 + c - c^2)\geq 1 \text. $$

    Therefore, given $\psi$ we get $c$, and then we choose $\phi=  a \,\psi + b\, \hat{\psi}$ such that $\alpha=1$ an $a=a_c$.
    %To get the value $1$ in \eqref{mmuc5} (or a value smaller than $1$) take $\alpha=1$ and solve the equation $ \,[    a^2 -   a  \, (a c + b d  )\, c +  (a c + b d  )^2] =0$  in the values $a,c.$ Then find $\psi$ and a corresponding $\phi$.
\end{proof}

 % The set of Hölder functions in  the kernel of the Ruelle operator $\ruelle$ (see \cite{LR1}) is generated by the orthonormal family, indexed by  words $w=(w_1, w_2,..,w_l)$, $l>1$,
 % \begin{equation}  \label{luc} \hat{a}_w=\frac{1}{\sqrt{2}} (\, e_{0,w} - \,\,e_{1,w}\,)=\frac{1}{\sqrt{2}} (\, e_{0,w_1, w_2,..,w_l} - \,\,e_{1,w_1, w_2,..,w_l}\,),
%\end{equation}
%plus the functions
 % $$ \hat{a}_{[\emptyset]}^0= \sqrt{2}\, (  \,\chi_{[00]} \,\, -  \,\,   \chi_{[10]}\,)$$
% \begin{equation} \label{tororo}
%\hat{a}_{[\emptyset]}^1=  \sqrt{2}\, (  \,\chi_{[11]} \,\, -  \,\,   \chi_{[01]}\,)\end{equation}

\begin{theorem} \label{treyw}
    Using the notation of \eqref{gel1} and \eqref{gel2}, take a non-constant
    $\psi=\sum_u b_u e_u + \beta_{0} e_{\varepsilon}^{0} + \beta_{1} e_{\varepsilon}^{1}\in \lp{2} (\bm{\mu}) $,
    such that $|\psi|=1$, and denote by $c(\psi)$ the value given by \eqref{lili}.

    Then,  if  $\hat{\psi}$ denotes the projection operator, we get the operator norm:
    \begin{equation} \label{podd1}
        1 \leq  \norm{\left[ \dirac , \pi(\hat{\psi}) \right]}  = \max{\left\{\norm{(\koopman \hat{\psi} - \hat{\psi}  \koopman)}, \norm{(\ruelle \hat{\psi} - \hat{\psi}  \ruelle)}\right\}}\leq \frac{3}{2\, \sqrt{2}}
    \end{equation}
    This is so because:
    \begin{equation} \label{odd1}
        \frac{3}{2\, \sqrt{2}} \geq  \norm{(\koopman \hat{\psi} - \hat{\psi}  \koopman)}  =
    \end{equation}
    \begin{equation} \label{req}
        = \sup_{\abs{\phi} = 1} \sqrt{\langle \phi, \psi \rangle^{2} - 2 \langle \phi, \psi \rangle \langle \koopman \phi, \psi \rangle \langle \koopman \psi, \psi \rangle + \langle \koopman \phi , \psi \rangle^{2}}
    \end{equation}
    \begin{equation} \label{bod1}
        \geq \sup_{w} \sqrt{ b_{w}^{2} + \frac{1}{2} \left( b_{0w} + b_{1w} \right)^{2} - 2 b_{w} \frac{1}{\sqrt{2}} \left( b_{0w} + b_{1w} \right) c(\psi) } \text,
    \end{equation}
    and:
    \begin{equation} \label{odd2}
        \frac{3}{2\, \sqrt{2}}  \geq   \norm{(\ruelle \hat{\psi} - \hat{\psi}  \ruelle)}  =
    \end{equation}
    \begin{equation} \label{req1}
        = \sup_{\abs{\phi} = 1} \sqrt{\langle \phi, \psi \rangle^{2} - 2 \langle \phi, \psi \rangle \langle \ruelle \phi, \psi \rangle \langle \ruelle \psi, \psi \rangle + \langle \ruelle \phi , \psi \rangle^{2}}
    \end{equation}
    \begin{equation} \label{bod2}
         \geq \sup_{w} \sqrt{ b_{w}^{2} + \frac{1}{2} b_{\sigma(w)}^{2} - 2 b_{w} \frac{1}{\sqrt{2}} b_{\sigma(w)} c(\psi) } \text.
    \end{equation}

    The inequality comes from Theorem \ref{Denk}.

    In terms of the elements of the basis, we get for $\phi$ of the form \eqref{gel2}
    $$ | \koopman \hat{\psi} (\phi) - \hat{\psi}\koopman (\phi)|^2= \langle \koopman \hat{\psi} (\phi) - \hat{\psi}\koopman (\phi), \koopman \hat{\psi} (\phi) - \hat{\psi}\koopman (\phi) \rangle=$$
    $$ \langle \phi,\psi \rangle^2 + ( \frac{1}{\sqrt{2}}  \sum_v a_v\,( b_{0 v} + b_{1 v})+\frac{1}{2} (  \alpha_1 + \alpha_0) (b_1 + b_0) +\frac{1}{2}  (\alpha_1- \alpha_0)  ( \beta_1 - \beta_0 )  )^2 $$
    $$-  2\,\, \langle \phi,\psi \rangle \,\{[\,\frac{1}{\sqrt{2}}  \sum_v a_v\,( b_{0 v} + b_{1 v})+\frac{1}{2} (  \alpha_1 + \alpha_0) (b_1 + b_0) +\frac{1}{2}  (\alpha_1- \alpha_0)  ( \beta_1 - \beta_0 )  \,]$$
    \begin{equation} \label{fret}
    [\sum_{l(u)>1}\frac{b_u}{\sqrt{2}}  (b_{0u}+  b_{1u}) + \frac{1}{2} (b_0 + b_1) (\beta_0 + \beta_1)  ]\} \text.
    \end{equation}
    For the proof see Proposition \ref{treywa} in the  Appendix.

    %Therefore, changing $\phi$,  from \eqref{lili4}, the maximal value of
    %$$ <\phi,\psi>^2 + <C_{\phi,\psi},\psi>^2 - 2 <\phi,\psi>\ <C_{\phi,\psi},\psi>\, <A_\psi,\psi> \,$$
    %is the value
    %$$ ( |<\phi,\phi>| + | <C_{\phi,\psi},\psi>|)^2\leq 2^2$$
    %\begin{equation} \label{frio3} \frac{1}{2} ( (a_1 - a_0) (\beta_1 - \beta_0)) - \frac{1}{2} ( (\alpha_1- \alpha_0) (b_1-  b_0) ) \}.
    %\end{equation}

    %Moreover,
    %\begin{equation} \label{geli1} \| \sqrt{\frac{1}{2}} (\koopman \hat{\psi} - \hat{\psi}  \koopman)\|\geq \sup_{l(w)>1} \frac{1}{2} b_w ( \,b_{ w_2,...,w_l}+ b_{0,w_1,w_2,...,w_l} +  b_{1,w_1,w_2,...,w_l}\,)    ,
    %\end{equation}
    % where we can assume all terms are positive.
\end{theorem}

\begin{theorem} \label{corte}
    Suppose $\psi$ with norm $1$ is a Hölder function in the kernel of the Ruelle operator $\ruelle$, then:
    \begin{equation} \label{krem}
        \norm{\left[ \dirac , \pi(\hat{\psi}) \right]}=\norm{(\koopman \hat{\psi} - \hat{\psi}  \koopman)} =\norm{(\ruelle \hat{\psi} - \hat{\psi}  \ruelle)}=1 \text.
    \end{equation}

    Moreover, take $\psi$ a Hölder function of the form $\psi = \koopman^k (f)$, $k\geq 1$, where
    $f$ is in the kernel of the Ruelle operator $\ruelle$, and has norm equal to $1$. Then
    $\norm{\left[ \dirac , \pi(\hat{\psi}) \right]}=1$.
\end{theorem}

\begin{proof}
    Given a function $\phi$ with norm $1$
    $$\langle \phi, \psi \rangle^{2} - 2 \langle \phi, \psi \rangle \langle \koopman \phi, \psi \rangle \langle \koopman \psi, \psi \rangle + \langle \koopman \phi , \psi \rangle^{2}=$$
    $$\langle \phi, \psi \rangle^{2} - 2 \langle \phi, \psi \rangle \langle \koopman \phi, \psi \rangle \langle \psi,\ruelle \psi \rangle + \langle \phi , \ruelle \psi \rangle^{2}=\langle \phi, \psi \rangle^{2}\leq 1.$$

    Taking $\phi=\psi$ above we get:
    $$\langle \psi, \psi \rangle^{2} - 2 \langle \psi, \psi \rangle \langle \koopman \psi, \psi \rangle \langle \koopman \psi, \psi \rangle + \langle \koopman \psi , \psi \rangle^{2}=1 \text.$$

    Therefore,
    $$ \norm{(\koopman \hat{\psi} - \hat{\psi}  \koopman)}  =$$
    \begin{equation} \label{reqew}
        \sup_{\abs{\phi} = 1} \sqrt{\langle \phi, \psi \rangle^{2} - 2 \langle \phi, \psi \rangle \langle \koopman \phi, \psi \rangle \langle \koopman \psi, \psi \rangle + \langle \koopman \phi , \psi \rangle^{2}} =1 \text.
    \end{equation}

    For the second claim, assume that  $\psi$ is a Hölder function of the form $\psi = \koopman^k (f)$,
    $k\geq 1$, where $f$  in the kernel of the Ruelle operator $\ruelle$ has norm equal to $1$. Then,
    $\psi = \koopman^k (f)$ also has norm equal to $1$.

    Then, taking $\phi = \koopman^k (f)=\psi$, we get:
    \begin{equation} \label{reqew91}
        \langle \koopman \phi, \psi \rangle =   \langle \koopman \koopman^k (f), \koopman^k (f)\rangle =\langle  \koopman (f), f\rangle=  \langle f, \ruelle(f)\rangle=0 \text.
    \end{equation}
    Therefore,
    $$  \langle \phi, \psi \rangle^{2} - 2 \langle \phi, \psi \rangle \langle \koopman \phi, \psi \rangle \langle \koopman \psi, \psi \rangle + \langle \koopman \phi , \psi \rangle^{2}=$$
    \begin{equation} \label{reqew92}
        \langle \phi, \psi \rangle^{2} + \langle \koopman \phi , \psi \rangle^{2}=\langle \phi, \psi \rangle^{2}=1 \text.
    \end{equation}

     Note that in  the notation of Theorem \ref{Denk} we get that:
     $$c= \cos (\theta)=   \langle  \koopman(\psi) ,\psi\rangle = 0 \text,$$
     and:
     $$\langle \phi, \psi \rangle^{2} - 2 \langle \phi, \psi \rangle \langle \koopman \phi, \psi \rangle \langle \koopman \psi, \psi \rangle + \langle \koopman \phi , \psi \rangle^{2}=$$
     \begin{equation} \label{mimo5}
         \alpha^2 \,[    a^2 -   a  \, (a c + b d  )\, c +  (a c + b d  )^2]= \alpha^2 \,[    a^2 +  ( b d  )^2] \text,
     \end{equation}
     where:
     \begin{equation} \label{mene}
         |\alpha|<1, a^2 + b^2=1=c^2 + d^2 \text.
     \end{equation}

     As $\alpha^2 \,[    a^2 +   b^2 d^2] \leq 1$, because $|d|\leq1$, we get \eqref{krem}.
\end{proof}

We denote $\mathfrak{N}$ the kernel of the Ruelle operator $\ruelle$. From Wold's Theorem (see
Section 7 in \cite{Pal}), it is known that the following orthogonal expression is true:
\begin{equation}
\lp{2}(\bm{\mu}) = \mathcal{H}_\infty \oplus \mathfrak{N} \oplus \koopman(\mathfrak{N}) \oplus \cdots \oplus \koopman^k(\mathfrak{N})\oplus \cdots
\end{equation}
where $\mathcal{H}_\infty$ is the set of almost everywhere constant functions.

As we will see one can get explicit values for the action of the Dirac operator on general projections.

\begin{proposition} \label{corte43}
    For Hölder functions $\psi$ with norm $1$ such that:
$$\psi \in\mathfrak{N} \oplus  \koopman(\mathfrak{N}) \oplus \cdots \oplus \koopman^k(\mathfrak{N})\oplus \cdots \text,$$
take $c=  \langle  \koopman(\psi) ,\psi\rangle$. Then:
\begin{equation}
    \norm{\left[ \dirac , \pi(\hat{\psi}) \right]}=\frac{1}{2} (2 + c - c^2) \text.
\end{equation}

The value $c$ can be obtained explicitly and in most of the cases $\frac{1}{2} (2 + c - c^2)>1$.
This will happen for a generic $\psi$.
\end{proposition}

\begin{proof}
    Assume $\psi$ with norm $1$ satisfies:
    \begin{equation} \label{pplo}
        \psi= \alpha_0  f_0 + \alpha_1 \koopman (f_1) +  \alpha_1 \koopman^2 (f_2)+ \cdots +  \alpha_k\koopman^k (f_k)+ \cdots \text,
    \end{equation}
    where $\sum_{k=0}^\infty \alpha_k^2=1$, and $|f_k|=1, \forall k$. By exchanging the signal of $f_k$ we can assume that all $\alpha_k\geq 0$.

    From Proposition \ref{Denk}, if $c=  \langle  \koopman(\psi) ,\psi\rangle>0$, then
    $\norm{\left[ \dirac , \pi(\hat{\psi}) \right]} > 1$.

    One can show that:
    $$c=\,\langle \koopman \psi, \psi \rangle \,= \alpha_0\alpha_1 \langle f_0,f_1 \rangle+\alpha_1 \alpha_2 \langle f_1,f_2 \rangle +\alpha_2 \alpha_3  \langle f_2,f_3 \rangle+ \cdots \text,$$
    and then we get one gets the value $\norm{\left[ \dirac , \pi(\hat{\psi}) \right]}$.

    Take for example $\psi= \alpha_0  f_0 + \alpha_1 \koopman (f_1)$, $\alpha_0\neq 0,1$, where
    $\langle f_0,f_1 \rangle >0.$ Then,
    $$c=\,\langle \koopman \psi, \psi \rangle \,= \alpha_0\alpha_1 \langle f_0,f_1 \rangle >0 \text.$$

    If in \eqref{pplo} we take $\langle f_{j+1},f_j\rangle=0$, $j\geq 0,$ we get $c=0$, and then
    $\norm{\left[ \dirac , \pi(\hat{\psi}) \right]} = 1 \text.$
\end{proof}

\subsection{Estimates in the case of elements in the Exel-Lopes $C^*$-algebra} \label{pos}

In this section we are interested in estimates of $ \|\,  \left[ \mathcal{D} , \pi (\mathcal{L} ) \right] \, \|$
for operators $\mathcal{L}$ of the form $\mathcal{L}=  \mult_{f}$ or $\mathcal{L}= \koopman^n \, \ruelle^n$, considered
in \cite{EL1} (see also \cite{EL2}). We are also interested in an interpretation of
$\|\, \left[ \mathcal{D} , \pi ( \mult_{f} ) \right] \,\|$ as a form of the supremum of a
discrete-time dynamical derivative of $f$. Note that \linebreak $|f - f \circ \sigma|_{\infty}\leq 1$
could be seen as a dynamical form of saying that $f$ has Lipschitz constant smaller than or equal to $1$.
We will explore here such a point of view.

We denote by $\mathfrak{B}$ the Borel sigma-algebra on $\{0,1\} ^\mathbb{N}$ and $\bm{\mu}$ the measure of
maximal entropy.

First, we consider multiplication operators of the form:
    \begin{align*}
        \mult_{f} : \lp{2} ( \bm{\mu} ) \longrightarrow & \lp{2} ( \bm{\mu} ) \\
        g \longmapsto & fg \text,
    \end{align*}
for a given continuous function $f:\{0,1\}^\mathbb{N} \to \mathbb{R}$ and measurable functions \(g \in \lp{2} ( \bm{\mu} )\).
Here, we will always use \(f \in C(\Omega)\) to denote a continuous function and \(g \in \lp{2} ( \bm{\mu} )\)
to denote a square-integrable function. We will use the expression \(\abs{\,\cdot\,}\) for the
\(\lp{2} ( \bm{\mu} )\)-norm and \(\abs{\,\cdot\,}_{\infty}\) for the supremum norm.

Among other things, we will be interested in a dynamical interpretation for the action of the Dirac
operator $\mathcal{D}$ which was introduced before; more precisely, for the action of the operator:
\begin{equation} \label{kuli}
    \left[ \mathcal{D} , \pi ( \mult_{f} ) \right] \text.
\end{equation}

As mentioned before, one important issue is to know when the norm of this operator (acting on
$\lp{2}(\bm{\mu}) \times \lp{2}(\bm{\mu})$) is smaller than or equal to $1$. Concerning the Connes distance, a kind
of normalization condition on $f$ would then be natural, for instance, $|f|_\infty=1$, or perhaps,
$f(0^\infty)=0$. But, we will not address this issue here.

One of our main results in this section is Theorem \ref{trocd} which claims:
\begin{theorem} \label{trocad}
For any $f \in C(\Omega)$:
\begin{equation} \label{lkor}
|f - f \circ \sigma|_{\infty}=    \abs{\koopman f - f}_{\infty}  \geq \norm{\left[ \mathcal{D} , \pi ( \mult_{f} ) \right]} \geq \abs{f - \ruelle f}_{\infty} \text.
\end{equation}

We will get equality on both sides when $f$ does not depend on the first coordinate.
\end{theorem}

Another important result is the estimate of Proposition \ref{erte}:

\begin{proposition} \label{erte}
For any \(f \in C(\Omega)\):
\begin{equation} \label{drms}
    \norm{\left[ \mathcal{D} , \pi ( \mult_{f} ) \right]}  = \abs{\sqrt{\ruelle \abs{\koopman f - f}^{2}}}_{\infty} \text.
\end{equation}
\end{proposition}

That is, instead of projection operators (the case analyzed in a previous section) we will now be 
interested in results for the multiplication operator \(\mult_{f}\). The operator $\mathcal{D}$
is dynamically defined, and then, any form of interpretation should also carry this structure. It
seems natural to us to presume that $f - f \circ \sigma$ could mean a derivative in a dynamical sense.
The next Lemma is in consonance with this point of view.

The expression:
\begin{equation} \label{forw}
    f - f \circ \sigma
\end{equation}
is a dynamical form of {\it discrete-time  forward derivative}.

Note that from \eqref{lkor}, if $|f - f \circ \sigma|_{\infty}\leq1$, then
$\norm{\left[ \mathcal{D} , \pi ( \mult_{f} ) \right]} \leq 1$.

Later in \eqref{drms1}  we will present a dynamical form of {\it discrete-time mean backward derivative}:
\begin{equation} \label{drms23}
    \sqrt{\frac{\abs{f(x) - f(0x)}^{2}}{2} + \frac{\abs{f(x) - f(1x)}^{2}}{2}} \text.
\end{equation}

In this direction in Proposition \ref{erte} we get the expression:
\begin{equation} \label{drms77}
    \norm{\left[ \mathcal{D} , \pi ( \mult_{f} ) \right]}  = \sup_{x \in \Omega} \sqrt{\frac{\abs{f(x) - f(0x)}^{2}}{2} + \frac{\abs{f(x) - f(1x)}^{2}}{2}} \text.
\end{equation}

\begin{lemma}
Given any $f \in C(\Omega)$, we get that \(\left[ \mathcal{D} , \pi ( \mult_{f} ) \right] = 0\)
is equivalent to \(f - f \circ \sigma=0\). The latter implies that $f$ is constant.
\end{lemma}

\begin{proof}
To prove this, first remember that since \(\mult_{f}\)
is self-adjoint:
\begin{align*}
\norm{\koopman \mult_{f} - \mult_{f} \koopman} = \norm{\ruelle \mult_{f} - \mult_{f} \ruelle} \text.
\end{align*}
Notice if \(f = f \circ \sigma\), then:
\begin{align*}
\koopman \mult_{f} (g) - \mult_{f} \koopman (g) & = \koopman (fg) - f \koopman g \\
                                                                  & = \left( \koopman f \right) \left( \koopman g \right) - f \koopman g \\
                                                                  & = \left( \koopman f - f \right) \koopman g \\
                                                                  & = 0 \text.
\end{align*}
Thus, \(\ruelle \mult_{f} - \mult_{f} \ruelle = 0\), and consequently
\(\left[ \mathcal{D} , \pi ( \mult_{f} ) \right] = 0\). On the other hand, if \(\left[ \mathcal{D} , \pi ( \mult_{f} ) \right] = 0\),
then:
\begin{align*}
\norm{\koopman \mult_{f} - \mult_{f} \koopman} = \norm{\ruelle \mult_{f} - \mult_{f} \ruelle} = 0 \text,
\end{align*}
and in particular:
\begin{align*}
|\koopman \mult_{f} (1) - \mult_{f} \koopman (1)| & =| \koopman f - f |\\
                                                                  & = 0 \text,
\end{align*}
which means \(f - f \circ \sigma=0\).
\end{proof}

The above motivates us to say that being constant with respect to \eqref{kuli} is the same as $f$
being constant. But, of course, not all continuous functions are invariant for the shift map, and
now we will analyze this more general family of functions.

\begin{remark} \label{yre1}
As \(\bm{\mu}\) is \(\sigma\)-invariant:
$$  \abs{\koopman g} = \left( \int \abs{g \circ \sigma}^{2} \d \bm{\mu} \right)^{\frac{1}{2}}
= \left( \int \abs{g}^{2} \d \bm{\mu} \right)^{\frac{1}{2}}
                     = \abs{g} \text.$$
This means that there exists a correspondence of the values on the left side (functions that are
$\sigma^{-1}( \mathfrak{B})$ measurable), and on the right side (functions that are $\mathfrak{B}$-measurable).

It also follows that:
\begin{align} \label{lio}
 \sup_{\abs{g} = 1} \abs{f \koopman g} & = \sup_{\abs{\koopman g} = 1} \abs{f \koopman g} \text.
\end{align}
\end{remark}

With the above Remark in mind, we look at the identity:
\begin{align} \label{commnorm}
 \norm{\left[ \mathcal{D} , \pi ( \mult_{f} ) \right]} & = \norm{\koopman \mult_{f} - \mult_{f} \koopman} \nonumber \\
                                                             & = \sup_{\abs{g} = 1} \abs{ \left( \koopman f - f \right) \koopman (g) } \text.
\end{align}

It is known that if $f$ is continuous:
\begin{align*}
 \sup_{\abs{g} = 1} \abs{f g}  = \abs{f}_{\infty} \text.
\end{align*}

It is also known that $\koopman\,\ruelle (f)$ is the conditional expectation of $f$ given the sigma-algebra $\sigma^{-1}( \mathfrak{B})$.

\begin{lemma}
For any \(f \in C(\Omega)\):
\begin{align} \label{inftyseminorm}
 \abs{f}_{\infty} \geq \sup_{\abs{g} = 1} \abs{f \koopman g} \geq \abs{\ruelle f}_{\infty}.
\end{align}
Furthermore, if \(f \in C(\Omega)\) does not depend on the first coordinate (that is, if \(f\) is \(\sigma^{-1}(\Sigma)\)-measurable), then all above  inequalities are equalities.
%or rather, if \(f = \koopman \ruelle f\))
\end{lemma}
\begin{proof}
First, we will show \eqref{inftyseminorm}. Note that from Jensen's inequality for conditional expectation:
\begin{align*}
 \abs{f}_{\infty} & = \sup_{\abs{g} = 1} \abs{f g} \\
              & \geq \sup_{\abs{\koopman g} = 1} \abs{f \koopman g} \\
              & = \sup_{\abs{g} = 1} \abs{f \koopman g} \\
              & = \sup_{\abs{g} = 1} \left( \int \abs{f \koopman g}^{2} \d \bm{\mu} \right)^{\frac{1}{2}} \\
              & = \sup_{\abs{g} = 1} \left( \int \koopman \ruelle \abs{f \koopman g}^{2} \d \bm{\mu} \right)^{\frac{1}{2}} \\
              & = \sup_{\abs{g} = 1} \left( \int \koopman \ruelle \left( \abs{f}^{2} \abs{\koopman g}^{2} \right) \d \bm{\mu} \right)^{\frac{1}{2}} \\
              & = \sup_{\abs{g} = 1} \left( \int \koopman \ruelle \left( \abs{f}^{2} \right) \abs{\koopman g}^{2} \d \bm{\mu} \right)^{\frac{1}{2}} \\
              & \geq \sup_{\abs{g} = 1} \left( \int \abs{\koopman \ruelle \left( f \right)}^{2} \abs{\koopman g}^{2} \d \bm{\mu} \right)^{\frac{1}{2}} \\
              & = \sup_{\abs{g} = 1} \left( \int \koopman \left( \abs{\ruelle \left( f \right)}^{2} \abs{g}^{2} \right) \d \bm{\mu} \right)^{\frac{1}{2}} \\
              & = \sup_{\abs{g} = 1} \left( \int \abs{\ruelle \left( f \right)}^{2} \abs{g}^{2} \d \bm{\mu} \right)^{\frac{1}{2}} \\
              & = \sup_{\abs{g} = 1} \left( \int \abs{\ruelle \left( f \right) g}^{2} \d \bm{\mu} \right)^{\frac{1}{2}} \\
              & = \sup_{\abs{g} = 1} \abs{\left( \ruelle f \right) g} \\
              & = \abs{\ruelle f}_{\infty} \text.
\end{align*}

Now, we will show the next claim. If $f$ does not depend on the first coordinate we get that
\(\abs{f}_{\infty} = \abs{\ruelle f}_{\infty}\).

Indeed, note that given any  $x=(x_1,x_2,\cdots)$, the two preimages $y_0=(0,x_1,x_2,\cdots),y_1=(1,x_1,x_2,\cdots)$
are such that $f(y_0)=f(y_1)$. Denote by $z=(a_1,a_2,\cdots, a_n,\cdots)$ the point where the continuous
function $f$ assumes the maximal value $\abs{f}_{\infty}$. Then,
$$\ruelle(f) (\sigma(z))=\frac{1}{2} ( f(1,a_2,a_3,\cdots, a_n,\cdots) + f(0,a_2,a_3,\cdots, a_n,\cdots) )=\abs{f}_{\infty} \text,$$
which is clearly the maximal value $\abs{\ruelle f}_{\infty}$.
\end{proof}

\begin{theorem} \label{trocd}
Replacing \(f\)  by \(\koopman f - f\) in \eqref{inftyseminorm}, in view of \eqref{commnorm},
we get for any \(f \in C(\Omega)\):
\begin{align*}
    \abs{\koopman f - f}_{\infty} & \geq \norm{\left[ \mathcal{D} , \pi ( \mult_{f} ) \right]} \geq \abs{f - \ruelle f}_{\infty} \text. \\
    \intertext{Moreover, if \(f\) does not depend on the first coordinate, then the same is true for \(\koopman f - f\),
    and we get the equalities:}
    \abs{\koopman f - f}_{\infty} & = \norm{\left[ \mathcal{D} , \pi ( \mult_{f} ) \right]} = \abs{f - \ruelle f}_{\infty} \text.
\end{align*}
\end{theorem}

\begin{proposition}
For any \(f \in C(\Omega)\):
\begin{equation}
    \norm{\left[ \mathcal{D} , \pi ( \mult_{f} ) \right]}  = \abs{\sqrt{\ruelle \abs{\koopman f - f}^{2}}}_{\infty}
\end{equation}

Expression \eqref{drms} can be written as:
\begin{equation} \label{drms1}
    \abs{\sqrt{\ruelle \abs{\koopman f - f}^{2}}}_{\infty}=\sup_{x \in \Omega} \sqrt{\frac{\abs{f(x) - f(0x)}^{2}}{2} + \frac{\abs{f(x) - f(1x)}^{2}}{2}} \text.
\end{equation}

The right-hand side of \eqref{drms1} is a form of the supremum of {\it mean backward derivative}.
\end{proposition}
\begin{proof}
We have:
\begin{align*}
    \sup_{\abs{g} = 1} \abs{f \koopman g} & = \sup_{\abs{g} = 1} \left( \int \abs{f \koopman g}^{2} \d \bm{\mu} \right)^{\frac{1}{2}} \\
                                             & = \sup_{\abs{g} = 1} \left( \int \abs{f}^{2} \abs{\koopman g}^{2} \d \bm{\mu} \right)^{\frac{1}{2}} \\
                                             & = \sup_{\abs{g} = 1} \left( \int \abs{f}^{2} \left( \koopman \abs{g}^{2} \right) \d \bm{\mu} \right)^{\frac{1}{2}} \\
                                             & = \sup_{\abs{g} = 1} \left( \int \left( \ruelle \abs{f}^{2} \right) \abs{g}^{2} \d \bm{\mu} \right)^{\frac{1}{2}} \\
                                             & = \sup_{\abs{g} = 1} \abs{\left( \sqrt{\ruelle \abs{f}^{2}} \right) g} \\
                                             & = \abs{\sqrt{\ruelle \abs{f}^{2}}}_{\infty} \text,
\end{align*}
then we substitute \(f\) for \(\koopman f - f\).
\end{proof}

\begin{corollary} \label{lolo}
For any \(f \in C(\Omega)\):
\begin{align*}
    \abs{\koopman f - f}_{\infty} \leq 1 \implies \norm{\left[ \mathcal{D} , \pi ( \mult_{f} ) \right]} \leq 1 \text.
\end{align*}
\end{corollary}

\begin{remark}
    The converse is not true. Take for example the function \linebreak \(f = \sqrt{2} \chi_{[0]} \in C(\Omega)\).
    Then:
    \begin{align*}
        \abs{\koopman f - f}_{\infty} & = \abs{\sqrt{2}(\chi_{[00]} + \chi_{[10]} - \chi_{[0]})}_{\infty} \\
                                        & = \abs{\sqrt{2}(\chi_{[10]} - \chi_{[01]})}_{\infty} \\
                                        & = \sqrt{2} > 1 \text.
    \end{align*}
    On the other hand, \eqref{drms} allows us to show that
    \(\norm{\left[ \mathcal{D} , \pi ( \mult_{f} ) \right]} = 1\). That is because:
    \begin{align*}
     \norm{\left[ \mathcal{D} , \pi ( \mult_{f} ) \right]} & = \sup_{x \in \Omega} \sqrt{\frac{\abs{f(x) - f(0x)}^{2}}{2} + \frac{\abs{f(x) - f(1x)}^{2}}{2}} \\
                                                                 & = \sup_{x \in \Omega} \sqrt{\frac{\sqrt{2}^{2}}{2}} \\
                                                                 & = 1 \text.
    \end{align*}
\end{remark}

Notice that expression \eqref{drms1} is a form RMS (root mean square), also called  quadratic mean,
and as such, a generalized mean in the sense of Kolmogorov (see \cite{Carva}), concerning the differences
\(\abs{f(x) - f(0x)}\) and \(\abs{f(x) - f(1x)}\). Thus, we may say the operator norm of the
``momentum'' \(\norm{\left[ \mathcal{D}, \pi ( \mult_{f} ) \right]}\) of a given continuous function 
\(f\), as we have previously defined, measures the supremum of a Kolmogorov
mean. In particular, it satisfies the following inequalities, which follow from the inequalities for
generalized means of different \textit{orders} as in \cite[Subsection 2.14.2; Theorem 1]{MV}. Here,
in the notation of \cite{MV} we are considering \textit{orders} \(- \infty\), \(-1\), \(0\), \(1\),
\(2\), and \(+ \infty\) respectively:

\begin{align*}
  \sup_{x \in \Omega} \min \left\{\begin{array}{l} \abs{f(x) - f(0x)}, \\ \quad \abs{f(x) - f(1x)}\end{array}\right\} & \leq \sup_{x \in \Omega} \frac{2}{\frac{1}{\abs{f(x) - f(0x)}} + \frac{1}{\abs{f(x) - f(1x)}}} \\
                                                                                                                      & \leq \sup_{x \in \Omega} \sqrt{\begin{array}{l} \abs{f(x) - f(0x)} \times \\ \quad \times \abs{f(x) - f(1x)}\end{array}} \\
                                                                                & \leq \sup_{x \in \Omega} \frac{\abs{f(x) - f(0x)}}{2} + \frac{\abs{f(x) - f(1x)}}{2} \\
                                                                                & \leq \norm{\left[ \mathcal{D} , \pi ( \mult_{f} ) \right]} \\
                                                                                & \leq \sup_{x \in \Omega} \max \left\{\begin{array}{l} \abs{f(x) - f(0x)}, \\ \quad \abs{f(x) - f(1x)}\end{array}\right\}  \\
                                                                                & = \abs{\koopman f - f}_{\infty} \text,
\end{align*}
plus in general for all \(-\infty < p \leq 2 \leq q < +\infty\):
\begin{align*}
  \sup_{x \in \Omega} \left( \begin{array}{l} \frac{\abs{f(x) - f(0x)}^{p}}{2} + \\ \quad + \frac{\abs{f(x) - f(1x)}^{p}}{2} \end{array} \right)^{\frac{1}{p}} & \leq \norm{\left[ \mathcal{D} , \pi ( \mult_{f} ) \right]} \\
                                                                                                                                                               & \leq \sup_{x \in \Omega} \left( \begin{array}{l} \frac{\abs{f(x) - f(0x)}^{q}}{2} + \\ \quad + \frac{\abs{f(x) - f(1x)}^{q}}{2} \end{array} \right)^{\frac{1}{q}} \text.
\end{align*}

The next result provides estimates for the \(\lp{2}(\bm{\mu})\)-norm (not for the uniform norm).

\begin{lemma}
For any \(f \in C(\Omega)\):
    \begin{align*}
        \norm{\left[ \mathcal{D} , \pi ( \mult_{f} ) \right]} \leq 1 \implies
        \begin{array}{c}
         \abs{ \koopman f - f } \leq 1 \\
         \& \\
         \abs{ \ruelle f - f } \leq 1 \text.
        \end{array}
    \end{align*}
\end{lemma}
\begin{proof}
Suppose \(\norm{\left[ \mathcal{D} , \pi ( \mult_{f} ) \right]} \leq 1\). This implies that:
\begin{align*}
    \norm{\koopman \mult_{f} - \mult_{f} \koopman} = \norm{\ruelle \mult_{f} - \mult_{f} \ruelle} \leq 1 \text.
\end{align*}
The constant function \(1 \in \lp{2} ( \bm{\mu} )\) has norm equal to \(1\). It follows that:
\begin{align*}
    \abs{ \koopman \mult_{f} (1) - \mult_{f} \koopman (1)} & = \abs{ \koopman f - f } \leq 1 \text, \\
    \intertext{and}
    \abs{ \ruelle \mult_{f} (1) - \mult_{f} \ruelle (1)} & = \abs{ \ruelle f - f } \leq 1 \text.
\end{align*}
\end{proof}

Now we will consider estimates of the value $ \norm{\left[\mathcal{D} , \pi(\mathcal{L}) \right]} $ for the class
of operators \(\mathcal{L}= \koopman^{n} \ruelle^{n}\), $n \geq 1$.

It is known that \(\mathcal{L}= \koopman^{n} \ruelle^{n}\) is the conditional expectation operator on
the sigma-algebra \(\sigma^{-n}(\mathfrak{B})\) (functions which do not depend on the \(n\) first
coordinates). First notice:
\begin{align*}
    \left[\mathcal{D} , \pi (\koopman^{n} \ruelle^{n})\right] & = \begin{pmatrix}
        0 & \koopman \koopman^{n} \ruelle^{n} - \koopman^{n} \ruelle^{n} \koopman \\
        \ruelle \koopman^{n} \ruelle^{n} - \koopman^{n} \ruelle^{n} \ruelle & 0 \\
    \end{pmatrix} \\
                                                               & = \begin{pmatrix}
        0 & \koopman \koopman^{n} \ruelle^{n} - \koopman^{n} \ruelle^{n-1} \\
        \koopman^{n-1} \ruelle^{n} - \koopman^{n} \ruelle^{n} \ruelle & 0 \\
    \end{pmatrix} \\
                                                               & = \begin{pmatrix}
                                                                   0 & \koopman \left( \koopman^{n} \ruelle^{n} - \koopman^{n-1} \ruelle^{n-1} \right) \\
                                                                   \left( \koopman^{n-1} \ruelle^{n-1} - \koopman^{n} \ruelle^{n} \right) \ruelle & 0 \\
    \end{pmatrix} \text.
\end{align*}
The operator \(\koopman^{n-1} \ruelle^{n-1} - \koopman^{n} \ruelle^{n}\) is the difference
between two projections, where the range of one is contained in the range of the other, and as such,
is also a projection.

Note that:
\begin{align*}
    \left( \koopman^{n-1} \ruelle^{n-1} - \koopman^{n} \ruelle^{n} \right)^{2} & = \begin{array}{l}
        \koopman^{n-1} \ruelle^{n-1} \koopman^{n-1} \ruelle^{n-1} - \koopman^{n-1} \ruelle^{n-1} \koopman^{n} \ruelle^{n} \\
        \quad - \koopman^{n} \ruelle^{n} \koopman^{n-1} \ruelle^{n-1} + \koopman^{n} \ruelle^{n} \koopman^{n} \ruelle^{n}
    \end{array} \\
                                                                                             & = \begin{array}{l}
        \koopman^{n-1} \ruelle^{n-1} - \koopman^{n} \ruelle^{n} \\
        \quad - \koopman^{n} \ruelle^{n} + \koopman^{n} \ruelle^{n}
    \end{array} \\
                                                                                             & = \koopman^{n-1} \ruelle^{n-1} - \koopman^{n} \ruelle^{n} \text.
\end{align*}

\begin{theorem} \label{EERT1} Given $n\geq 1$
\begin{equation} \label{EERT}
    \norm{\left[\mathcal{D} , \pi( \koopman^{n} \ruelle^{n})\right]}=1 \text.
\end{equation}
\end{theorem}

\begin{proof}
We know that \(\koopman\) preserves inner products, so that for any bounded operator \(A\):
\begin{align*}
    \norm{\koopman A} & = \sup_{\abs{g} = 1} \abs{\koopman A g} \\
                         & = \sup_{\abs{g} = 1} \sqrt{\langle \koopman A g , \koopman A g \rangle} \\
                         & = \sup_{\abs{g} = 1} \sqrt{\langle A g , A g \rangle} \\
                         & = \sup_{\abs{g} = 1} \abs{A g} \\
                         & = \norm{A} \text.
\end{align*}
Substituting \(A\) for \(\koopman^{n} \ruelle^{n} - \koopman^{n-1} \ruelle^{n-1}\) we get:
\begin{align*}
    \norm{\left[\mathcal{D} , \pi (\koopman^{n} \ruelle^{n})\right]} & = \norm{\left( \koopman^{n-1} \ruelle^{n-1} - \koopman^{n} \ruelle^{n} \right) \ruelle} \\
                                                                      & = \norm{\koopman \left( \koopman^{n} \ruelle^{n} - \koopman^{n-1} \ruelle^{n-1} \right)} \\
                                                                      & = \norm{\koopman^{n} \ruelle^{n} - \koopman^{n-1} \ruelle^{n-1}} \\
                                                                      & = \norm{\koopman^{n-1} \ruelle^{n-1} - \koopman^{n} \ruelle^{n}} \\
                                                                      & = 1 \text.
\end{align*}

\end{proof}

 \section{A generalized boson formalism for \\ the maximal entropy probability} \label{bobo}

In this section we elaborate on the meaning of a dynamical version of generalized boson systems which
we introduced before. We present the computations that are required for the justification of several
claims presented in the Introduction section \ref{iintoc}. We will describe in detail different estimates
that will corroborate our claims for this setting and the differences and similarities with respect
to the non-dynamical point of view. The main results are summarized in Propositions \ref{ELV} and \ref{lalau}.

%\begin{remark} \label{pupu} $\ruelle:\lp{2}(m)\to \lp{2}(m)$ is surjective. Indeed,
%if $G$,  denotes the Kernel of the Koopman operator  $\ruelle^{\dagger}=\koopman:\lp{2}(m)\to \lp{2}(m)$, we know that $G=\{0\}$, and then we get that $G^\perp=  \lp{2}(m)$. Denoting Im $(\ruelle)=F$ we know that $F= G^\perp=  \lp{2}(m)$ (see \cite{Co}).

%\end{remark}

%We will assume from now on that $m=\bm{\mu}$ is the maximal entropy measure on $\Omega= \{0,1\}^\mathbb{N}$, which means $J=1/2.$

\subsection{A brief note on fermion formalism}

Denote $\hat{f}= \frac{1}{\sqrt{2}}\ruelle$ and $\hat{f}^\dag= \frac{1}{\sqrt{2}}\koopman$.
The main CAR relation should be $\{\hat{f},\hat{f}^\dag\}=I$ (it will not be true  here).

Assume that $J$ is continuous, positive, and satisfies for any $x$
$$ \sum_{\sigma(y)= x} J( y) =1.$$

We call $J$ the Jacobian.

The Ruelle operator $\ruelle_{\log J}$ acts on continuous functions $\phi$ and is defined by:
$$\ruelle_{\log J}(\phi)(x) =\sum_{\sigma(y)= x} J( y) \phi(y) \text.$$

$\phi \to (\ruelle_{\log J} \circ  \koopman) (\phi)$ is the identity.

We want to obtain results similar to the ones in \cite{CHLS}.

Note that $\koopman \ruelle (f)$ is the conditional expectation of $f$ given the sigma-algebra $\sigma^{-1} (\mathfrak{B})$, where  $\mathfrak{B}$ is the Borel sigma-algebra in $\{0,1\}^\mathbb{N}.$ Therefore, if $f$  does not depend on the first coordinate we get:
\begin{equation} \label{conex}
\koopman \ruelle (f)=f \text.
\end{equation}

However, as we mentioned before, the CAR relations are not true: indeed, for any $x$ and any $\phi$
$$ \{ \ruelle_{\log J},  \koopman\}(\phi)(x)= (\ruelle_{\log J}  \koopman + \koopman \ruelle_{\log J}) (\phi)(x)= \phi(x) + \sum_{\sigma(y)= \sigma(x)} J( y) \phi(y) \text.$$

This means:
$$\{ \hat{f},  \hat{f}^\dag\}(\phi)(x)=  \frac{1}{2}(\phi(x) + \sum_{\sigma(y)= \sigma(x)} J( y) \phi(y)) \neq \phi(x) \text.$$

Note that in the case $\phi$ does not depend on the first coordinate we get that:
\begin{equation} \label{pote1}
    \{ \hat{f},  \hat{f}^\dag\}(\phi)=\phi \text.
\end{equation}

Is it natural to consider bounded operators acting on the $\lp{2}(\bm{\mu})$ space $\mathcal{F}$ of the functions
$\phi$ that do not depend on the first coordinate. Then, CAR:
\begin{equation}\label{feo1}\{ \hat{f},  \hat{f}^\dag\}= I
\end{equation}
is true on $\mathcal{F}$.

Nonetheless, for the general case of  $\phi$
$$ \int \{ \ruelle,  \koopman\}(\phi) \d \bm{\mu}= 2\, \int \phi \d \bm{\mu} \text.$$

Then, we get:
\begin{equation}\label{po1} \int \{ \hat{f},  \hat{f}^\dag\}(\phi) \d \bm{\mu}=  \int \phi \d \bm{\mu} \text.
\end{equation}

In this way $\{ \hat{f},  \hat{f}^\dag\}\neq I$, but anyway, when we consider the  action of integrating
functions in the $\lp{2}(\bm{\mu})$ space, we get something similar to CAR.

Note that:
\begin{equation} \label{poet2}
    \left[ \ruelle_{\log J},  \koopman\right] (\phi)(x)= (\ruelle_{\log J}  \koopman - \koopman \ruelle_{\log J}) (\phi)(x)= \phi(x) - \sum_{\sigma(y)= \sigma(x)} J( y) \phi(y) \text.
\end{equation}

\subsection{Boson formalism}

Now we write \(B \defn 2^{-\frac{1}{2}} \ruelle\) so that \(B^{\dagger} = 2^{-\frac{1}{2}} \koopman\).

For the boson formalism the Number operator $B B^{\dagger}$ acts on the $\lp{2}(\bm{\mu})$ space.

Note that for any $x$ and any $\phi$
$$[B, B^{\dagger} ] (\phi)(x)= \frac{1}{2} \,[ \ruelle_{\log J},  \koopman] (\phi)(x)= \frac{1}{2} \,(\ruelle_{\log J}  \koopman - \koopman \ruelle_{\log J}) (\phi)(x)= $$
$$\frac{1}{2} \,(\,\phi(x) - [ L ( \phi) \circ \sigma] (x)\,) =\frac{1}{2} \,(\phi(x) - \sum_{\sigma(y) =\sigma(x)} J( y) \phi(y)) \neq \phi(x).$$

In this direction, the CCR $\left[B,  B^{\dagger}\right]=1$  is  not true, however, a generalized boson
form of CCR, as  in \cite{Kuo}, will be considered in Proposition \ref{lalau}.

Here we take $f(n) = 2^{-\frac{n}{2}}$ for the dynamical generalized boson system we consider.

\begin{lemma} \label{Boscf} For $\phi \in L^2 (\mu)$, we get that  $2\, [B, B^{\dagger} ] (\phi)=[L, K] (\phi)$ is the projection of $\phi$ on the Kernel of $L.$

\end{lemma} 

\begin{proof} First note that $L( v \circ \sigma))=v$ for all continuous function $v:\Omega \to \mathbb{R}.$  In this way $ L (  [ L ( \phi) \circ \sigma] )=  L ( \phi).$ 
 Therefore,
$$ L (\,[B, B^{\dagger} ] (\phi)\,) =\frac{1}{2} L (  \,\,\phi - [ L ( \phi) \circ \sigma] \,) = \frac{1}{2} ( L ( \phi  )- L (\phi))=0. $$

This means that $ [B, B^{\dagger} ] (\phi) $ is in the kernel of the Ruelle operator.

It is known from  Section 4 in \cite{LTF} that the subspace orthogonal to the kernel of $L$ is the set of functions of the form $b \circ \sigma$, where $b$ is in $L^2 (\mu)$. 

Then,
$$2\,  \,[B, B^{\dagger} ] (\phi)\, = \phi - u,$$
where $u$ is orthogonal to the kernel of $L.$ This means that
$$\phi = 2\,  \,[B, B^{\dagger} ] (\phi)\,+\, u,$$
and the claim follows.

\end{proof}

From now on $\Omega =\{0,1\}^\mathbb{N}$, $J=\frac{1}{2}$ and $\bm{\mu}$ is the measure of maximal entropy
for the shift $\sigma$.

Denote by \(W\) the set of finite words $w = w_{1} w_{2} w_{3} \cdots w_{l}$ on the symbols $\{0,1\}$.

Given a finite word $w = w_{1} w_{2} w_{3} \cdots w_{l}$, denote by $\ell(w) = l$ the size of the word $w$.
We denote by $[w]=[w_{1} w_{2} w_{3} \cdots w_{l}]$ the associated cylinder set.

$W^\ast$ denotes the set of finite words $w$ with size $\ell(w) \geq 1$. By abuse of notation, we will
say that $\sigma(w_{1} w_{2} w_{3} \cdots w_{l})= w_{2} w_{3} \cdots w_{l}$ when $l \geq 2$.

In Appendix B in \cite{CHLS}, adapting  Theorem 3.5 in \cite{KS}  to our case, it was shown that:
\begin{align*}
  \mathbb{B} \defn \left\lbrace \vphantom{e_{\varepsilon}^{0}}{e_w} \mid w \in \textstyle W^\ast \right\rbrace \cup \left\lbrace  e_{\varepsilon}^{0} ,   e_{\varepsilon}^{1} \right\rbrace
\end{align*}
is a Haar basis of $\lp{2}(\bm{\mu})$, where, in the partricular case of \(J = \frac{1}{2}\), the elements \(e_{w}\) take the form:
\begin{align*}
  e_{w} & \defn \frac{1}{\sqrt{\bm{\mu}([w])}} \left( \sqrt{ \tfrac{\bm{\mu}([w0])}{\bm{\mu}([w1])}} \chi_{[w1]}-\sqrt{  \tfrac{\bm{\mu}([w1])}{\bm{\mu}([w0])}} \chi_{[w0]} \right) \\
        & = 2^{\frac{\ell(w)}{2}} \left( \chi_{[w1]} - \chi_{[w0]} \right) \text.
\end{align*}
and the other two elements take the form:
\begin{align*}
  e_{\varepsilon}^{0} & \defn -\frac{1}{ \sqrt{\bm{\mu} ([0])}} \chi_{[0]} \\
                      & = - 2^{\frac{1}{2}} \chi_{[0]} \text, \\ \intertext{and:}
  e_{\varepsilon}^{1} & \defn \frac{1}{ \sqrt{\bm{\mu} ([1])}} \chi_{[1]} \\
                      & = 2^{\frac{1}{2}} \chi_{[1]} \text. \\ 
\end{align*}

We say that $\tilde{w} = a_1 a_2 \cdots a_r$ is a prefix of $w = b_1 b_2 \cdots b_s$ if $r\leq s$ and
$w = a_1 a_2 \cdots a_r b_{r+1} \cdots b_s$. We use the notation $\tilde{w}\leq w$. Note that if $w$ is
not a prefix of $\tilde{w}$ and also $\tilde{w}$ is not a prefix of $w$, then the product of the
functions $e_{\tilde{w}}\,e_{w}=0$.

If $u<v$, meaning \(u \leq v\) and \(u \neq v\), we get:

\begin{equation} \label{fu1}
e_u \,e_v=  \sqrt{2^{l(u)}} e_v= - (-1)^{v_{l(u)+1}} \sqrt{2^{l(u)}} e_v \text.
\end{equation}

Moreover, if $u=v$, then:
\begin{equation} \label{fu3}
e_u \,e_u=2^{l(u)}\ \left( \chi_{[u0]} + \chi_{[u1]} \right)= 2^{l(u)}  \chi_{[u]} \text.
\end{equation}

The relations \eqref{fu1} and \eqref{fu3} will play an important role in this section. From \cite{LR1}  (see also \cite{LR2}) we get for the Ruelle operator:

\begin{proposition} \label{esta}
Given $w = w_1 w_2 \cdots w_n$ with a size (strictly) larger than $1$.
\begin{equation} \label{lwi115}
  B ( e_{w_1 w_2 \cdots w_n} )  = 2^{- 1} e_{w_2 w_3 \cdots w_n} = 2^{- 1} e_{\sigma(w)} \text.
\end{equation}
Moreover,
\begin{equation} \label{lui15}
  B ( e_{1} )  = B ( e_{0} ) = 2^{- 1} (\chi_{[1]}- \chi_{[0]})= 2^{- 1} \left( 2^{- \frac{1}{2}}(e_{\varepsilon}^1 + e_{\varepsilon}^0) \right) \text,
\end{equation}
and:
\begin{align*}
  B( e_{\varepsilon}^{0} ) & = -2^{- 1} \\
  B( e_{\varepsilon}^{1} ) & =  2^{- 1} \\
\end{align*}
\begin{equation} \label{lui55} B (2^{- \frac{1}{2}}(e_{\varepsilon}^1 + e_{\varepsilon}^0) )=0.
\end{equation}
\end{proposition}

Similarly, we get for the Koopman operator:

\begin{proposition} \label{kesta}

Given $w = w_1 w_2 \cdots w_n$ with a size at least $1$.
\begin{equation} \label{klui115}
  B^{\dagger} ( e_{w_1 w_2 \cdots w_n} )  = 2^{- 1} \left( e_{0 w_1 w_2 \cdots w_{n}} + e_{1 w_{1} w_2 \cdots w_n} \right) = 2^{- 1} \left( e_{0w} + e_{1w} \right) \text.
\end{equation}
Moreover,
\begin{align} \label{klui15}
  B^{\dagger} ( e_{\varepsilon}^{0} ) & = -(\chi_{[00]} + \chi_{[10]})  \nonumber \\
  B^{\dagger} ( e_{\varepsilon}^{1} ) & =  (\chi_{[01]} + \chi_{[11]}) 
\end{align}
so that:
\begin{equation} \label{klui55} B^{\dagger} (2^{- \frac{1}{2}}(e_{\varepsilon}^1 + e_{\varepsilon}^0) )= 2^{- 1} \left(e_{[0]} + e_{[1]}\right) \text.
\end{equation}
\end{proposition}

    Now define \(| n \rangle \defn 2^{-\frac{n}{2}} \sum_{\ell(w) = n} e_{w}\) and
    \(| 0 \rangle \defn 2^{\frac{1}{2}}(e_{\varepsilon}^{0} + e_{\varepsilon}^{1})\). Notice that \(\langle n , n \rangle = \langle 0 , 0 \rangle = 1\).

    We have:
    \begin{align*}
      B^{\dagger} | n \rangle & = 2^{-\frac{1}{2}}| n + 1 \rangle \text{, and:} \\
      B | n \rangle & = 2^{-\frac{1}{2}} | n - 1 \rangle \text;
    \end{align*}
    and consequently, for \(n \geq 1\):
    \begin{align*}
        B^{\dagger n} | 0 \rangle & = 2^{-\frac{n}{2}} | n \rangle \text. \\
    \end{align*}

In our setting $| 0 \rangle$ corresponds to the vacuum.

Moreover,
    $$    B^{\dagger} B | n \rangle = 2^{-1} | n \rangle.$$

    In fact, define \(|n, w \rangle \defn 2^{-\frac{n+1}{2}} \left( \sum_{\ell(u) = n} e_{u0w} - \sum_{\ell(u) = n} e_{u1w} \right)\)
    and \linebreak \(|0, w \rangle \defn 2^{-\frac{1}{2}} \left( e_{0w} - e_{1w} \right)\). \textbf{Then:}
    \begin{align}
      B^{\dagger} |n, w \rangle & = 2^{-\frac{1}{2}} |n+1, w \rangle \label{lui17} \\
      B |n, w \rangle & = 2^{-\frac{1}{2}} |n-1, w \rangle \label{elui1591} \\
      B |0, w \rangle & = 0 \nonumber \text.
    \end{align}
    \textbf{Moreover,}
    \begin{equation*} 
      B B^\dagger |n \rangle = 2^{-1}|n \rangle \text. 
    \end{equation*}
    \textbf{Therefore,}
    \begin{align*}
      \left[ B, B^{\dagger} \right]|n \rangle & = 0 \quad \forall \, n \geq 1 \\
      \left[ B, B^{\dagger} \right]|0 \rangle & = | 0 \rangle \text.
    \end{align*}
    
    Now consider the index set \(W^{\star} \defn \left\{\star\right\} \cup W\). Define \(|n, \star \rangle \defn |n \rangle\).
    The family \(\left\{|0, w\rangle \mid w \in W^{\star}\right\}\) is an orthonormal basis for the kernel of the Ruelle
    operator (see \cite{LR1}). Also:
    \begin{proposition} \label{ELV}
        For any \(n \geq 1\):
    \begin{align}
        B^{\dagger n} |0, w \rangle & = 2^{-\frac{n}{2}} |n, w \rangle \text{, and:} \label{ELV1} \\
        B^{\dagger} B |n, w \rangle & = 2^{-1} |n, w \rangle \label{ELV2} \text.
    \end{align}
    \end{proposition}
    Finally, by Wold's decomposition theorem,
    \begin{align*}
        \left\{\chi_{[\varepsilon]} = 1\right\} \cup \left\{|n, w\rangle \mid n \in \natural, w \in W^{\star}\right\}
    \end{align*}
    is an orthonormal basis for \(\lp{p}(\bm{\mu})\); therefore:
    \begin{proposition} \label{lalau}
      The commutator of the Ruelle and Koopman operators satisfies:
      \begin{align} \label{lalau1}
        \left[ B, B^{\dagger} \right] = 2^{-1} \sum_{w \in W^{\star}} |0, w\rangle \langle 0, w| \text.
      \end{align}
      
       $\left[ B, B^{\dagger} \right]$ is the projection on the Kernel of the Ruelle Operator $L$.
       
    \end{proposition}

    Because of all of the above, we say the Ruelle-Koopman pair constitutes a generalized boson system (in the sense of \cite{Kuo}). Accordingly,
    we say that \(B^{\dagger}\) is a generalized creation operator, \(B\) is a generalized annihilation operator, and
    \(B^{\dagger} B\) is a generalized number operator.

In the notation of \cite{Kuo}  (where  functions  \(F,f: \natural \cup \left\{0\right\} \longrightarrow \real\) play a fundamental role), the  associated function \(f\) satisfies \(f(n) = 2^{-\frac{n}{2}}\),  moreover, \(F\) is such that \(F(0) = 2^{-1}\), and \(F(n) = 0\) for all \(n \geq 1\).

    An alternative point of view can be considered, which is in some sense in consonance with part of the reasoning in \cite{Kuo}; consider an orthonormal family in the kernel of the Ruelle operator (as for instance in \cite{LR1}), denote the index set by \(\Lambda\) and denote the family by \(\left\{| i \rangle \mid i \in \Lambda\right\}\). Define 
    the vectors \(|n, i \rangle \defn \koopman^{n} | i \rangle\)\footnote{We employ a different notation from \cite{Kuo} 
    and write \(|n, i\rangle\) instead of \(|n_{i}\rangle\).} and projections \(P_{i} \defn \sum | n, i\rangle\langle n, i| \). 
    Then, the family \(b_{i} \defn P_{i} B P_{i}\) constitutes a generalized boson system.

    Finally, in yet another alternative point of view, we may quotient the space \(\lp{2}(\bm{\mu})\) by the kernel of the Ruelle operator plus the space of almost-everywhere constant functions  \(\ker L \oplus \complex1\). This quotient is isomorphic to the space of square-integrable sequences \(\ell^{2}(\complex)\). The Ruelle operator passes to the quotient and its 
    action becomes the action of the shift operator (of square-integrable sequences) \(T : \ell^{2}(\complex) 
    \longrightarrow \ell^{2}(\complex)\). Then, there is a single bosonic mode \(b = 2^{- \frac{1}{2}} T\) and the vacuum 
    state is \((1,0,0,\cdots) \in \ell^{2}(\complex)\).

    \begin{remark}
      By virtue of Wold's decomposition, any isometry of a Hilbert space is a generalized boson.
    \end{remark}

\begin{example} \label{lour}

It is possible to consider an analogy between our study and what is observed in the harmonic oscillator.
Denote by $|n\rangle$ the \(n\)th eigenfunction of the quantized operator ${\bf H}$ associated to the
classical Harmonic oscillator Hamiltonian $H(x,p)=  \frac{\hat{p}^2}{2 m} + \frac{m\, w^2 \, \hat{x}^2 }{2} $.
Denote by $\hat{x}$ and $\hat{p}$, respectively, the position and momentum operator acting on $\lp{2}(\d x)$,
for the Lebesgue measure $d x$ on  $\mathbb{R}$.

The creation  operator $\mathcal{C}$ and the annihilation operator $\mathcal{A}$ (see Section 11 in \cite{Hall}) are given by:
 \begin{equation} \label{ouuy}
     \mathcal{C} = \sqrt{\frac{m\, w \,  }{2\, \hbar}} \,\, \hat{x} + i\, \sqrt{\frac{1}{2 m w \, \hbar }}\,\, \,\hat{p}= \sqrt{\frac{m\, w \,  }{2\, \hbar}} \,\, \hat{x} + \, \sqrt{\frac{1}{2 m w \, \hbar }}\,\frac{\d}{\d x}\,
 \end{equation}
and:
\begin{equation} \label{ouuy0}
    \mathcal{A} = \sqrt{\frac{m\, w \,  }{2\, \hbar}} \,\, \hat{x} - i\, \sqrt{\frac{1}{2 m w \, \hbar }}\,\, \hat{p}= \sqrt{\frac{m\, w \,  }{2\, \hbar}} \,\, \hat{x} - \, \sqrt{\frac{1}{2 m w \, \hbar }}\,\,\frac{\d}{\d x} \text.
\end{equation}

\index{creation operator}

\index{annihilation operator}

From the property $\left[\hat{p},\hat{x}\right]= -i\, \hbar \,I$, we can get the Canonical Commutation Relation:
$$ \left[\mathcal{C}, \mathcal{A}\right]= \, I \text.$$

In this case, for $\mathcal{C}$ and $\mathcal{A}=\mathcal{C}^{\dagger}$, respectively, the creation and
annihilation operators, we get:
$$ \mathcal{C} (|n\rangle) = \sqrt{n+1}\,\,\,  |(n +1)\rangle$$
and:
$$ \mathcal{A} (|n\rangle) = \sqrt{n}\,\,\,  |(n - 1)\rangle \text.$$

From \eqref{sor}, this corresponds to the case where $f(n)= \sqrt{n!}$.

In this case, $|0\rangle$ corresponds to the eigenfunction associated with the smallest eigenvalue of the Hamiltonian operator ${\bf H}$, and:
\begin{equation} \label{vvac111p}
    \mathcal{A} (|0\rangle)=0 \text.
\end{equation}
For this reason $|0\rangle$ is called the vacuum (eigenfunction) for the quantized harmonic
oscillator (see \eqref{vacc0}).
\end{example}

\section{Appendix}

In the Appendix we will present the proofs of several claims we mentioned before.

\begin{proposition} \label{lal1a}
  Suppose $e_w =e_{w_1 w_2 \cdots w_n}$, $l(w)>1$, then:
  \begin{align*}
      (\koopman \hat{e}_w - \hat{e}_w  \koopman) (\phi) & =
\sqrt{\frac{1}{2}}  [ \, (\int e_w\, \phi \d \bm{\mu} )  ( e_{0w} + e_{1w} ) -  \int e_{w_2 w_3 \cdots w_n}\,\phi \d \bm{\mu} \,\,\,e_w \,\,] \\
                                                              & = \frac{1}{\sqrt{2}} \left[ \langle e_{w} , \phi \rangle \left( e_{0w} + e_{1w} \right) - \langle e_{\sigma(w)} , \phi \rangle e_{w} \right] \text, \\
                                                              \intertext{and}
      (\ruelle \hat{e}_w - \hat{e}_w  \ruelle) (\phi) & =
\sqrt{\frac{1}{2}} \,[ \int e_w\,\phi \d \bm{\mu}\,\,  e_{w_2 \cdots w_n}   - \, \int ( e_{0w} + e_{1w} )\,\phi \d \bm{\mu}\,  \, \, \,e_w\,] \\
                                                              & = \frac{1}{\sqrt{2}} \left[ \langle e_{w} , \phi \rangle e_{\sigma(w)} - \langle e_{0w} + e_{1w} , \phi \rangle e_{w} \right]\text. \\
                                                              \intertext{Moreover,}
      (\koopman \hat{e}_{ e_{\varepsilon}^0} - \hat{e}_{ e_{\varepsilon}^0}  \koopman) (\phi) & =-\sqrt{2} \,\langle e_{\varepsilon}^0 , \phi \rangle ( \chi_{00} + \chi_{10}) + \langle e_{\varepsilon}^0 , (\phi \circ \sigma) \rangle e_{\varepsilon}^0 \text, \\
      (\koopman \hat{e}_{ e_{\varepsilon}^1} - \hat{e}_{ e_{\varepsilon}^1}  \koopman) (\phi) & =\sqrt{2} \,\langle e_{\varepsilon}^1 , \phi \rangle ( \chi_{01} + \chi_{11})- \langle e_{\varepsilon}^1 , (\phi \circ \sigma) \rangle e_{\varepsilon}^1 \text, \\
      (\ruelle \hat{e}_{ e_{\varepsilon}^0} - \hat{e}_{ e_{\varepsilon}^0}  \ruelle) (\phi) & =-\sqrt{\frac{1}{2}} \,\langle e_{\varepsilon}^0 , \phi \rangle + \sqrt{2} \langle ( \chi_{00} + \chi_{10})  ,\phi \rangle \,e_{\varepsilon}^0 \text, \\
             \intertext{and}
      (\ruelle \hat{e}_{ e_{\varepsilon}^1} - \hat{e}_{ e_{\varepsilon}^1}  \ruelle) (\phi) & = \sqrt{\frac{1}{2}}\, \langle e_{\varepsilon}^1 , \phi \rangle - \sqrt{2} \langle ( \chi_{01} + \chi_{11})  ,\phi \rangle \,e_{\varepsilon}^1 \text.
  \end{align*}
\end{proposition}

\begin{proof}
    Note that $\forall x$
    \begin{align*}
        (\koopman \hat{e}_w - \hat{e}_w  \koopman) (\phi)(x) & = \koopman \left( \langle e_{w} , \phi \rangle e_{w} \right) - \left( \langle e_{w} , \koopman \phi \rangle e_{w} \right) \\
                                                                   & = \left( \langle e_{w} , \phi \rangle \koopman e_{w} \right) - \left( \langle \ruelle e_{w} , \phi \rangle e_{w} \right) \\
                                                                   & = \langle e_{w} , \phi \rangle \left( \frac{1}{\sqrt{2}} \left( e_{0w} + e_{1w} \right) \right) - \left( \langle \ruelle e_{w} , \phi \rangle e_{w} \right) \\
                                                                   & = \langle e_{w} , \phi \rangle \left( \frac{1}{\sqrt{2}} \left( e_{0w} + e_{1w} \right) \right) - \left( \langle \frac{1}{\sqrt{2}} e_{w_{2} w_{3} \cdots w_{n} } , \phi \rangle e_{w} \right) \\
                                                                   & = \sqrt{\frac{1}{2}}  [ \, (\int e_w\, \phi \d \bm{\mu} )  ( e_{0w} + e_{1w} ) -  \int e_{w_2 w_3 \cdots w_n}\,\phi \d \bm{\mu} \,\,\,e_w \,\,] \text.
    \end{align*}

    Moreover, $\forall x$
    \begin{align*}
        (\ruelle \hat{e}_w - \hat{e}_w  \ruelle) (\phi)(x) & = \ruelle \left( \langle e_{w} , \phi \rangle e_{w} \right) - \left( \langle e_{w} , \ruelle \phi \rangle e_{w} \right) \\
                                                                   & = \left( \langle e_{w} , \phi \rangle \ruelle e_{w} \right) - \left( \langle \koopman e_{w} , \phi \rangle e_{w} \right) \\
                                                                   & = \langle e_{w} , \phi \rangle \left( \frac{1}{\sqrt{2}} e_{w_{2} w_{3} \cdots w_{n} } \right) - \left( \langle  \frac{1}{\sqrt{2}} \left( e_{0w} + e_{1w} \right)  , \phi \rangle e_{w} \right) \\
                                                                   & = \sqrt{\frac{1}{2}} \,[ \int e_w\,\phi \d \bm{\mu}\,\,  e_{w_2 \cdots w_n}   - \, \int ( e_{0w} + e_{1w} )\,\phi \d \bm{\mu}\,  \, \, \,e_w\,] \text.
    \end{align*}
\end{proof}

\begin{proposition} \label{okla}
    For a word $w$ satisfying $l(w)>1$, given a  generic word $\tilde{w}$, and the corresponding
    element $e_{\tilde{w}}$, we get that:
    \begin{equation} \label{frio1aaa}
        | (\koopman \hat{e}_w - \hat{e}_w  \koopman) (e_{\tilde{w}})|^2= \int [\, (\koopman \hat{e}_w - \hat{e}_w  \koopman) (e_{\tilde{w}} )\,]^2 \d \bm{\mu}
    \end{equation}
    is equal to $0$, if  $w\neq \tilde{w}\neq w_2 w_3 \cdots w_n$, is equal to $1/2$ if
    $w\neq\tilde{w}= w_2 w_3 \cdots w_n$, and is equal to 1 if $\tilde{w}=w$. Moreover,
    \begin{equation}
        | (\ruelle \hat{e}_w - \hat{e}_w  \ruelle) (e_{\tilde{w}})|^2= \int [\, (\ruelle \hat{e}_w - \hat{e}_w  \ruelle) (e_{\tilde{w}} )\,]^2 \d \bm{\mu}
    \end{equation}
    is equal to $0$, if  $\tilde{w} \neq w \neq \sigma{\tilde{w}}$, and is equal to $1/2$ if either $\tilde{w}\neq w = \sigma{\tilde{w}}$ or $\tilde{w}=w$.
\end{proposition}

\begin{proof}
    For a fixed word $w$ such that $l(w)>1$, given a generic word $\tilde{w}\neq w = w_1 w_2 w_3 \cdots w_n$,
    and the corresponding element $e_{\tilde{w}}$, we get:
    \begin{align*}
        (\koopman \hat{e}_w - \hat{e}_w  \koopman) (e_{\tilde{w}}) & = \frac{1}{\sqrt{2}} \left[ \langle e_{w} , e_{\tilde{w}} \rangle \left( e_{0w} + e_{1w} \right) - \langle e_{\sigma(w)} , e_{\tilde{w}} \rangle e_{w} \right] \\
                                                                         & = 0 - \frac{1}{\sqrt{2}} \langle e_{\sigma(w)} , e_{\tilde{w}} \rangle e_{w} \\
                                                                         & = 0 - \sqrt{\frac{1}{2}} \int e_{w_2 w_3 \cdots w_n}\, e_{\tilde{w} } \d \bm{\mu} \,\,\,e_w \text.
    \end{align*}

    When, $w= \tilde{w}$ we get:
    \begin{equation} \label{sac1}
        (\koopman \hat{e}_w - \hat{e}_w  \koopman) (e_w)= \sqrt{\frac{1}{2}}  [ \,  (\, e_{0w1} + e_{1w1} )-  (e_{0w0} + e_{1w0} )\,] \text.
    \end{equation}

    When, $w\neq \tilde{w}\neq w_2 w_3 \cdots w_n$ we get:
    \begin{equation} \label{sac2}
        (\koopman \hat{e}_w - \hat{e}_w  \koopman) (e_{\tilde{w}})=0 \text.
    \end{equation}

    When, $w\neq\tilde{w} = w_2 w_3 \cdots w_n$ we get:
    \begin{equation} \label{sac3}
        (\koopman \hat{e}_w - \hat{e}_w  \koopman) (e_{\tilde{w}})=- \sqrt{\frac{1}{2}}\, 2^{n-1} \, \int   \chi_{\tilde{w}} \d \bm{\mu} \,e_w=- \sqrt{\frac{1}{2}}\, e_w \text.
    \end{equation}

    If $l(w)>2$ we get that

    \begin{equation} \label{sac4}
        (\koopman \hat{e}_w - \hat{e}_w  \koopman) (e_1)=0 =  (\koopman \hat{e}_w - \hat{e}_w  \koopman) (e_0) \text.
    \end{equation}

    It is also true that for $w=01$ and $w=11$
    \begin{equation} \label{sac5}
        (\koopman \hat{e}_w - \hat{e}_w  \koopman) (e_1)= - \sqrt{\frac{1}{2}}\, e_w \text,
    \end{equation}
    and, for $w=10$ and $w=00$
    \begin{equation} \label{sac6}
        (\koopman \hat{e}_w - \hat{e}_w  \koopman) (e_1)= 0 \text.
    \end{equation}

    Moreover, for $w=10$ and $w=00$
    \begin{equation} \label{sac7}
        (\koopman \hat{e}_w - \hat{e}_w  \koopman) (e_0)= - \sqrt{\frac{1}{2}}\, e_w \text,
    \end{equation}

    and, for $w=01$ and $w=11$
    \begin{equation} \label{sac8}
        (\koopman \hat{e}_w - \hat{e}_w  \koopman) (e_0)= 0 \text.
    \end{equation}

    Note  that for  a finite word  $w= w_1 w_2 w_3 \cdots w_n$, where $l(w)>1$
    \begin{equation} \label{frio2}
        (\koopman \hat{e}_w - \hat{e}_w  \koopman) (e_{\varepsilon}^0)=0 = (\koopman \hat{e}_w - \hat{e}_w  \koopman) (e_{\varepsilon}^1) \text.
    \end{equation}

    With regards to \(\ruelle \hat{e}_{w} - \hat{e}_{w} \ruelle\), \(l(w) > 1\) note that:
    \begin{enumerate}
    \item if \(\tilde{w} = w\), then:
        \begin{align*}
        \left( \ruelle \hat{e}_{w} - \hat{e}_{w} \ruelle \right) e_{\tilde{w}} = \frac{1}{\sqrt{2}} e_{\sigma(w)} \text,
        \end{align*}
    \item if \(\tilde{w} = 0w\) or \(\tilde{w} = 1w\) (in other words, if \(\sigma(\tilde{w}) = w\)), then:
        \begin{align*}
        \left( \ruelle \hat{e}_{w} - \hat{e}_{w} \ruelle \right) e_{\tilde{w}} = - \frac{1}{\sqrt{2}} e_{w} \text,
        \end{align*}
    \item else:
        \begin{align*}
        \left( \ruelle \hat{e}_{w} - \hat{e}_{w} \ruelle \right) e_{\tilde{w}} = 0 \text.
        \end{align*}
    \end{enumerate}
\end{proof}

\begin{theorem} \label{aguaa}
    Given a fixed word  $w= w_1 w_2 w_3 \cdots w_n$, $l(w)\geq2$, we get for the operator norm:
    \begin{equation} \label{frio3a}
        \norm{(\koopman \hat{e}_w - \hat{e}_w  \koopman)} = 1 \text,
    \end{equation}
    and also:
    \begin{equation}
        \norm{(\ruelle \hat{e}_w - \hat{e}_w  \ruelle)} = 1 \text;
    \end{equation}
    therefore:
    \begin{equation} \label{frio4a}
        \norm{\left[\mathcal{D}, \pi(\hat{e}_w)\right]} = 1 \text.
    \end{equation}
\end{theorem}

\begin{proof}
    For getting \eqref{frio4a} we will use \eqref{kju}. Consider
    $\phi = \sum_u a_u e_u + \alpha_0 e_{\varepsilon}^0 +\alpha_1 e_{\varepsilon}^1$ in such way that
    $\sum_u |a_u|^2+ |\alpha_0|^2 + |\alpha_1|^2 =1$. Take $\tilde{w}=w_2 w_3 \cdots w_n$.

    \begin{align*}
        \int \left( (\koopman \hat{e}_w - \hat{e}_w  \koopman) (\phi)\right)^{2} \d \bm{\mu}
        & = \int \left( (\koopman \hat{e}_w - \hat{e}_w  \koopman) (a_w e_w + a_{\tilde{w}}e_{\tilde{w} } ) \right)^{2} \d \bm{\mu} \\
        & = \int \left( a_{w} \frac{1}{\sqrt{2}} \left( e_{0w} + e_{1w} \right) - a_{\tilde{w}} \frac{1}{\sqrt{2}} e_{w} \right)^{2} \d \bm{\mu} \\
        & = \frac{1}{2} \int a_{w}^{2} e_{0w}^{2} + a_{w}^{2} e_{1w}^{2} + a_{\tilde{w}}^{2} e_{w}^{2} \d \bm{\mu} \\
        & = a_{w}^{2} + \frac{1}{2} a_{\tilde{w}}^{2} \\
        & \leq \left( a_{w}^{2} + a_{\tilde{w}}^{2} \right) \\
        & \leq 1
    \end{align*}

    If we take $\phi= e_w$ we get the maximal value, which is \(1\).

    Besides,
    \begin{align*}
        \int \left( (\ruelle \hat{e}_w - \hat{e}_w  \ruelle) (\phi)\right)^{2} \d \bm{\mu}
        & = \int \left( (\ruelle \hat{e}_w - \hat{e}_w  \ruelle) (a_w e_w + a_{0w} e_{0w} + a_{1w} e_{1w} ) \right)^{2} \d \bm{\mu} \\
        & = \int \left( a_{w} \frac{1}{\sqrt{2}} e_{\sigma(w)} - \frac{1}{\sqrt{2}} \left( a_{0w} + a_{1w} \right) e_{w} \right)^{2} \d \bm{\mu} \\
        & = \frac{1}{2} \int a_{w}^{2} e_{\sigma(w)}^{2} + \left( a_{0w} + a_{1w} \right)^{2} e_{w}^{2} \d \bm{\mu} \\
        & = \frac{1}{2} \left( a_{w}^{2} + \left( a_{0w} + a_{1w} \right)^{2} \right) \\
        & \leq \frac{1}{2} \left( a_{w}^{2} + 2 \left( a_{0w}^{2} + a_{1w}^{2} \right) \right) \\
        & \leq \left( a_{w}^{2} + a_{0w}^{2} + a_{1w}^{2} \right) \\
        & \leq 1
    \end{align*}

    If we take \(\phi = \frac{1}{\sqrt{2}} \left(e_{0w} + e_{1w}\right)\) we get the maximal value, which is \(1\).

    Now consider \(\phi = \sum_u a_u e_u\) and \(\psi = \sum_u b_u e_u\) such that:
    \begin{align*}
        \norm{(\phi, \psi)}^{2} & = \norm{\phi}^{2} + \norm{\psi}^{2} \\
                                & = \sum_{u} \abs{a_{u}}^{2} + \sum_{u} \abs{b_{u}}^{2} \\
                                & = \sum_{u} \abs{a_{u}}^{2} + \abs{b_{u}}^{2} \\
                                & \leq 1 \text.
    \end{align*}
    and compute the squared norm of \(\left[ \dirac, \pi(\hat{e}_{w}) \right]\) applied to it:
    \begin{align*}
        \norm{\left[ \dirac, \pi(\hat{e}_{w}) \right]\left(\phi, \psi\right)}^{2} & = \int \left( (\koopman \hat{e}_w - \hat{e}_w  \koopman) (\psi)\right)^{2} \d \bm{\mu} \\
                                                      & \quad \quad + \int \left( (\ruelle \hat{e}_w - \hat{e}_w  \ruelle) (\phi)\right)^{2} \d \bm{\mu} \\
                                                      & = b_{w}^{2} + \frac{1}{2} b_{\tilde{w}}^{2} + \frac{1}{2} \left( a_{w}^{2} + \left( a_{0w} + a_{1w} \right)^{2} \right) \\
                                                      & \leq  b_{w}^{2} + b_{\tilde{w}}^{2} + a_{w}^{2} + a_{0w}^{2} + a_{1w}^{2}  \\
                                                      & \leq 1 \text.
    \end{align*}
    Now, if we take either the pair \(\left(0, e_{w}\right)\) or the pair \(\left(\frac{1}{\sqrt{2}} \left(e_{0w} + e_{1w}\right), 0\right)\),
    we get the maximal value, which is \(1\).
\end{proof}

\begin{theorem} \label{treywa}
    Using the notation of \eqref{gel1} and \eqref{gel2}, take a non-constant
    $\psi=\sum_u b_u e_u + \beta_{0} e_{\varepsilon}^{0} + \beta_{1} e_{\varepsilon}^{1}\in \lp{2} (\bm{\mu}) $,
    such that $|\psi|=1$, and denote by $c(\psi)$ the value given by \eqref{lili}.

    Then, if $\hat{\psi}$ denotes the projection operator, we get the operator norm:
    \begin{align*}
         1\leq  \norm{\left[ \dirac , \pi(\hat{\psi}) \right]} & = \max{\left\{\norm{(\koopman \hat{\psi} - \hat{\psi}  \koopman)}, \norm{(\ruelle \hat{\psi} - \hat{\psi}  \ruelle)}\right\}}\leq \frac{3}{2\, \sqrt{2}} \text.
    \end{align*}
    This is so because:
    \begin{equation} \label{odd1aa}
        \frac{3}{2\, \sqrt{2}} \geq  \norm{(\koopman \hat{\psi} - \hat{\psi}  \koopman)}  =
    \end{equation}
    \begin{equation} \label{reqaa}
        \sup_{\abs{\phi} = 1} \sqrt{\langle \phi, \psi \rangle^{2} - 2 \langle \phi, \psi \rangle \langle \koopman \phi, \psi \rangle \langle \koopman \psi, \psi \rangle + \langle \koopman \phi , \psi \rangle^{2}}
    \end{equation}
    \begin{equation} \label{bod1a}
        \geq \sup_{w} \sqrt{ b_{w}^{2} + \frac{1}{2} \left( b_{0w} + b_{1w} \right)^{2} - 2 b_{w} \frac{1}{\sqrt{2}} \left( b_{0w} + b_{1w} \right) c(\psi) } \text,
    \end{equation}
    and:
    \begin{equation} \label{odd2a}
        \frac{3}{2\, \sqrt{2}} \geq  \norm{(\ruelle \hat{\psi} - \hat{\psi}  \ruelle)}  =
    \end{equation}
    \begin{equation} \label{req1aa}
        \sup_{\abs{\phi} = 1} \sqrt{\langle \phi, \psi \rangle^{2} - 2 \langle \phi, \psi \rangle \langle \ruelle \phi, \psi \rangle \langle \ruelle \psi, \psi \rangle + \langle \ruelle \phi , \psi \rangle^{2}}
    \end{equation}
    \begin{equation} \label{bod2bb}
        \geq \sup_{w} \sqrt{ b_{w}^{2} + \frac{1}{2} b_{\sigma(w)}^{2} - 2 b_{w} \frac{1}{\sqrt{2}} b_{\sigma(w)} c(\psi) } \text.
    \end{equation}

    The inequality comes from Proposition \ref{Denk}.

    In terms of the elements of the basis, we get for $\phi $ of the form \eqref{gel2}
    $$ | \koopman \hat{\psi} (\phi) - \hat{\psi}\koopman (\phi)|^2= \langle \koopman \hat{\psi} (\phi) - \hat{\psi}\koopman (\phi), \koopman \hat{\psi} (\phi) - \hat{\psi}\koopman (\phi) \rangle=$$
    $$ \langle \phi, \psi \rangle^2 + ( \frac{1}{\sqrt{2}}  \sum_v a_v\,( b_{0 v} + b_{1 v})+\frac{1}{2} (  \alpha_1 + \alpha_0) (b_1 + b_0) +\frac{1}{2}  (\alpha_1- \alpha_0)  ( \beta_1 - \beta_0 )  )^2 $$
    $$-  2\,\, \langle \phi, \psi \rangle\,\{[\,\frac{1}{\sqrt{2}}  \sum_v a_v\,( b_{0 v} + b_{1 v})+\frac{1}{2} (  \alpha_1 + \alpha_0) (b_1 + b_0) +\frac{1}{2}  (\alpha_1- \alpha_0)  ( \beta_1 - \beta_0 )  \,]$$
    \begin{equation} \label{fretaa}
        [\sum_{l(u)>1}\frac{b_u}{\sqrt{2}}  (b_{0u}+  b_{1u}) + \frac{1}{2} (b_0 + b_1) (\beta_0 + \beta_1)  ]\} \text.
    \end{equation}
    %Therefore, changing $\phi$,  from \eqref{lili4}, the maximal value of
    %$$ <\phi,\psi>^2 + <C_{\phi,\psi},\psi>^2 - 2 <\phi,\psi>\ <C_{\phi,\psi},\psi>\, <A_\psi,\psi> \,$$
    %is the value
    %$$ ( |<\phi,\phi>| + | <C_{\phi,\psi},\psi>|)^2\leq 2^2$$
    %\begin{equation} \label{frio3} \frac{1}{2} ( (a_1 - a_0) (\beta_1 - \beta_0)) - \frac{1}{2} ( (\alpha_1- \alpha_0) (b_1-  b_0) ) \}.
    %\end{equation}

    %Moreover,
    %\begin{equation} \label{geli1} \| \sqrt{\frac{1}{2}} (\koopman \hat{\psi} - \hat{\psi}  \koopman)\|\geq \sup_{l(w)>1} \frac{1}{2} b_w ( \,b_{ w_2,...,w_l}+ b_{0,w_1,w_2,...,w_l} +  b_{1,w_1,w_2,...,w_l}\,)    ,
    %\end{equation}
    % where we can assume all terms are positive.
\end{theorem}

\begin{proof}
    Given a fixed $\psi$, the equality to $\frac{3}{2 \sqrt{2}}$ in \eqref{odd1} is true because the
    Koopman operator satisfies the hypothesis of Proposition \ref{Denk}. Equality \eqref{odd2} follows
    from duality. The expressions \eqref{req} and \eqref{req1} will follow from \eqref{fretu} and \eqref{rew}.

    First note that:
    $$ \langle \koopman (\phi), \psi) \rangle =$$
    $$ \langle (\sum_v a_v \frac{1}{\sqrt{2}}  ( e_{0v} + e_{1v} )- \alpha_0\sqrt{2}( \chi_{00} + \chi_{10} ))+ \sqrt{2} \alpha_1( \chi_{01} + \chi_{11} ) ,$$
    $$  (  \sum_u b_u e_u + \beta_0 e_{\varepsilon}^0 +\beta_1 e_{\varepsilon}^1  ) \rangle=$$
    $$\frac{1}{\sqrt{2}}  \sum_v a_v\,( b_{0 v} + b_{1 v})+\frac{1}{2} ( \alpha_0 b_0 + \alpha_0 b_1 + \alpha_1 b_0+\alpha_1 b_1    \,) +\frac{1}{2} (\alpha_0 \beta_0 - \alpha_0 \beta_1 +\alpha_1 \beta_1 -\alpha_1 \beta_0 ) \text.$$
    From \eqref{klui115}, \eqref{klui15}, and \eqref{klui55}
    $$\koopman \hat{\psi} (\phi)=\koopman ( \langle \phi,\psi \rangle \,  \sum_u b_u e_u + \beta_0 e_{\varepsilon}^0 +\beta_1 e_{\varepsilon}^1)=$$
    $$\langle \phi, \psi \rangle\,  [\sum_u b_u \frac{1}{\sqrt{2}}  ( e_{0u} + e_{1u} )- \sqrt{2}\beta_0( \chi_{00} + \chi_{10} )+ \sqrt{2}  \beta_1(\chi_{01} + \chi_{11} )] \text, $$
    and:
    $$\hat{\psi}\koopman (\phi)= \langle \koopman (\phi), \psi) \rangle (  \sum_u b_u e_u + \beta_0 e_{\varepsilon}^0 +\beta_1 e_{\varepsilon}^1  )=$$
    $$[\,\frac{1}{\sqrt{2}}  \sum_v a_v\,( b_{0 v} + b_{1 v})+\frac{1}{2} (  \alpha_1 + \alpha_0) (b_1 + b_0) +\frac{1}{2}  (\alpha_1- \alpha_0)  ( \beta_1 - \beta_0 )  \,] $$
    $$ (  \sum_u b_u e_u + \beta_0 e_{\varepsilon}^0 +\beta_1 e_{\varepsilon}^1  ) \text.$$

    Then, given $\phi$ such that $|\phi|=1$
    $$ \koopman \hat{\psi} (\phi) - \hat{\psi}\koopman (\phi)=$$
    $$ \langle \phi, \psi \rangle\, [ \sum_u b_u \frac{1}{\sqrt{2}}  ( e_{0u} + e_{1u} )- \sqrt{2}\beta_0( \chi_{00} + \chi_{10} )+  \sqrt{2}\beta_1( \chi_{01} + \chi_{11} )] $$
    $$ -[\,\frac{1}{\sqrt{2}}  \sum_v a_v\,( b_{0 v} + b_{1 v})+\frac{1}{2} (  \alpha_1 + \alpha_0) (b_1 + b_0) +\frac{1}{2}  (\alpha_1- \alpha_0)  ( \beta_1 - \beta_0 )  \,]$$
    $$  (  \sum_u b_u e_u + \beta_0 e_{\varepsilon}^0 +\beta_1 e_{\varepsilon}^1  ) \text.$$

    We leave it for the reader to calculate expression \eqref{fret}.

    Alternatively, note that:
    $$
    \langle \langle \phi, \psi \rangle \koopman \psi , \langle \phi, \psi \rangle \koopman \psi \rangle  = \langle \phi , \psi \rangle^{2} \langle \koopman \psi , \koopman \psi \rangle \\
                                                                                                               = \langle \phi, \psi \rangle^{2} \text,$$

    and:
    $$\langle \langle \koopman \phi , \psi \rangle \psi , \langle \koopman \phi , \psi \rangle \psi \rangle  = \langle \koopman \phi , \psi \rangle^{2} \langle \psi , \psi \rangle
    \langle \koopman \phi , \psi \rangle^{2} \text.$$

    Moreover,
    \begin{align*}
    2 \langle \langle \phi, \psi \rangle \koopman \psi , - \langle \koopman \phi , \psi \rangle \psi \rangle & = - 2 \langle \phi, \psi \rangle \langle \koopman \phi, \psi \rangle \langle \koopman \psi, \psi \rangle \text.
    \end{align*}

    %Notice:
    %\begin{equation} \label{uuti}| \langle \koopman \psi , \psi \rangle \leq \abs{\koopman \psi} \abs{\psi} = 1.
    %\end{equation}

    Finally,
    $$    \langle \koopman \hat{\psi} (\phi) - \hat{\psi} \koopman (\phi), \koopman \hat{\psi} (\phi) - \hat{\psi} \koopman (\phi) \rangle =$$
    $$
    \langle \koopman \hat{\psi} (\phi) , \koopman \hat{\psi} (\phi) \rangle - 2 \langle \koopman \hat{\psi} (\phi) , \hat{\psi} \koopman (\phi) \rangle + \langle \hat{\psi} \koopman (\phi) , \hat{\psi} \koopman (\phi) \rangle $$
    \begin{equation} \label{fretu}
        \langle \phi, \psi \rangle^{2} - 2 \langle \phi, \psi \rangle \langle \koopman \phi, \psi \rangle \langle \koopman \psi, \psi \rangle + \langle \koopman \phi , \psi \rangle^{2} \text.
    \end{equation}

    Similarly,  note that:
    $$
    \langle \ruelle \hat{\psi} (\phi) - \hat{\psi} \ruelle (\phi), \ruelle \hat{\psi} (\phi) - \hat{\psi} \ruelle (\phi) \rangle  =$$
    $$   \langle \ruelle \hat{\psi} (\phi) , \ruelle \hat{\psi} (\phi) \rangle - 2 \langle \ruelle \hat{\psi} (\phi) , \hat{\psi} \ruelle (\phi) \rangle + \langle \hat{\psi} \ruelle (\phi) , \hat{\psi} \ruelle (\phi) \rangle= $$
    \begin{equation} \label{rew}
        \langle \phi, \psi \rangle^{2} - 2 \langle \phi, \psi \rangle \langle \ruelle \phi, \psi \rangle \langle \ruelle \psi, \psi \rangle + \langle \ruelle \phi , \psi \rangle^{2} \text.
    \end{equation}

    Applying the right hand-side of \eqref{fret} to $\phi=e_w$, we get:
    \begin{equation} \label{bod}
        b_{w}^{2} + \frac{1}{2} \left( b_{0w} + b_{1w} \right)^{2} - 2 b_{w} \frac{1}{\sqrt{2}} \left( b_{0w} + b_{1w} \right) c(\psi) \text.
    \end{equation}

    Expression:
    $$
      \norm{(\koopman \hat{\psi} - \hat{\psi}  \koopman)}  \geq  $$
    \begin{equation}\label{trr} \sup_{w} \sqrt{
    b_{w}^{2} + \frac{1}{2} \left( b_{0w} + b_{1w} \right)^{2} - 2 b_{w} \frac{1}{\sqrt{2}} \left( b_{0w} + b_{1w} \right) c(\psi)
    }
    \end{equation}
    follows from \eqref{bod}.
\end{proof}

\medskip

This work is part of the PhD thesis of William M. M. Braucks in Programa de Pos-gradua\c c\~ao em Matem\' atica (see \cite{rkboson})  - UFRGS (2025)

\bigskip

  William M. M. Braucks  (braucks.w@gmail.com)
  \smallskip

   Artur O. Lopes (arturoscar.lopes@gmail.com)
  \smallskip

  Inst. Mat. Est. - UFRGS - Porto Alegre, Brazil


\begin{thebibliography}{99}

\bibitem{Arai}
A. Arai,
Analysis on Fock Spaces and Mathematical Theory of Quantum Fields, World Scient (2017)

\bibitem{Pal}
S. Bezuglyi and  P. E. T. Jorgensen,
Transfer Operators, Endomorphisms, and Measurable Partitions, Lect. Notes in Mat. 2217, Springer Verlag


\bibitem{rk-lp}
W. M. M. Braucks and A. O. Lopes, The Dirac operator for the Ruelle-Koopman pair on \texorpdfstring{$\lp{p}$}{Lp}-spaces: an interplay between Connes distance and symbolic dynamics, arXiv (2025)



\bibitem{rkboson}
W.~M.~M. Braucks. The {Dirac} operator for the Ruelle-Koopman pair: an interplay between Connes distance and symbolic dynamics. PhD thesis, P\'os-gradua\c c\~ao  em   Matem\'atica - UFRGS (2025)

\bibitem{Bra}  S. Braibant, G. Giacomelli and  M. Spurio,
Particles and Fundamental Interactions: An Introduction to Particle Physics, Springer Verlag (2012)


\bibitem{BruW} J-B Bru 
and W. de S Pedra,
$C^*$-Algebras and Mathematical Foundations of Quantum Statistical
Mechanics, Springer Verlag (2023)

\bibitem{Carva}
M. de Carvalho,
Mean, What do you Mean?, The American Statistician Volume 70,  Issue 3, 270-274 (2016)

\bibitem{Cian}
Z-P Cian, G. Zhu, S-K Chu, A. Seif, W. DeGottardi, L. Jiang, and M. Hafezi,
Photon Pair Condensation by Engineered Dissipation, Phys. Rev. Lett. 123, 063602 -2019

\bibitem{CHLS}
L. Cioletti, L. Hataishi. A. O. Lopes and M. Stadlbauert,
Spectral Triples on Thermodynamic Formalism and Dixmier Trace Representations of Gibbs Measures: theory and examples, Nonlinearity, Volume 37, Number 10, 105010 (56pp) (2024)

\bibitem{cfrconnes}
A. Connes,
Compact metric spaces, Fredholm modules, and hyperfiniteness. Ergodic Theory Dynam. Systems, 9(2):207-220, 1989.

\bibitem{Connes2}
A. Connes,
Noncommutative geometry. Academic Press, Inc., San Diego, CA, 1994

\bibitem{Co}
J. Conway,
A Course in Functional Analysis, Springer Verlag, 1990

\bibitem{EL2}
R. Exel and A. O. Lopes,
$C^*$-Algebras, approximately proper equivalence relations and Thermodynamic Formalism, Vol 24, pp 1051--1082 Erg Theo and Dyn Syst (2004).


\bibitem{EL1}
R. Exel and A. O. Lopes,
$C^*$-Algebras and Thermodynamic Formalism, Sao Paulo Journal of Mathematical Sciences 2, 1 (2008), 285-307


\bibitem{Ex1}
R. Exel,
KMS states for generalized gauge actions on Cuntz-Krieger algebras (An application of the Ruell-Perron-Frobenius Theorem). Bol. Soc. Brasil. Mat.  35(2004), no. 1, 1-12.

\bibitem{ExVe}
R. Exel and A. Vershik,
$C^*$-Algebras of Irreversible Dynamical Systems, Canad. J. Math. Vol. 58 (1), 2006 pp. 39-63

\bibitem{Hall}
B. C. Hall,
Quantum theory for mathematicians, Springer Verlag (2013)

\bibitem{JuPu}
A. Julien and I. Putnam.
Spectral triples for subshifts. J. Funct. Anal., 270(3):1031–1063, 2016

\bibitem{KS}
M. Kessebohmer and T. Samuel,
Spectral metric spaces for Gibbs measures, Journal of Functional Analysis, 265, 1801–-1828 (2013)

\bibitem{KSS}
M. Kessebohmer, M. Stadlbauer and B. O.  Stratmann,
Lyapunov spectra for KMS states on Cuntz-Krieger algebras. Math. Z. 256 (4): 871-893 (2007).

\bibitem{Kolmo}
A. Kolmogorov,
Sur la Notion de la Moyenne, Atti della Academia Nazionale dei Lincei, 12, 323-343 (1930)

\bibitem{Kuo}
E-J Kuo,Y Xu, D.Hangleiter, A. Grankin and M. Hafezi,
Boson sampling for generalized bosons, Physical Review Research, 4, 043096 (2022)

\bibitem{LCH}
B-S Lin, H-L Chen and T-H Heng,
Note on Connes spectral distance of qubits, arXiv (2022)

\bibitem{LQM} A. O. Lopes, 
Uma breve Introducao a Matematica da Mecanica Quantica,
XXXI Coloquio Brasileiro de Matematica  (2017) 

\bibitem{LR1}
A. O. Lopes and R. Ruggiero,
The sectional curvature of the infinite dimensional manifold of  H\"older equilibrium probabilities,  to appear in  Proc. of the Edinburg Matematical Soc.


\bibitem{LR2}  A. O. Lopes and R. Ruggiero,
Geodesics and dynamical information projections on the manifold of H\"older equilibrium probabilities, to appear in Stoch. and Dyn.



\bibitem{LTF} A. O. Lopes, Thermodynamic Formalism, Maximizing Probabilities
and Large Deviations, notes online (UFRGS)

http://mat.ufrgs.br/$\sim$alopes/pub3/notesformteherm.pdf

\bibitem{MV}
D. S. Mitrinovic and P. M. Vasic,
Analytic inequalities. Vol. 1. Berlin: Springer, 1970.

\bibitem{PP}
W. Parry and M.  Pollicott.
Zeta functions and the periodic orbit structure of hyperbolic dynamics, \emph{Ast\'erisque} Vol {187-188} 1990


\bibitem{PS}
M. Pollicott and R. Sharp, A Weil–-Petersson type metric on spaces of metric graphs.
Geom. Dedicata 172, 229–-244 (2014)


\bibitem{Potts}
P. P. Potts,
Quantum Thermodynamics, arxiv (2024)

\bibitem{Sha1}
R. Sharp.
Spectral triples and Gibbs measures for expanding maps on Cantor sets. J. Noncommut. Geom., 6(4):801-817 2012.

\bibitem{Sha2}
R. Sharp.
Conformal Markov systems, Patterson-Sullivan measure on limit sets and spectral triples. Discrete Contin. Dyn. Syst., 36(5):2711-2727, 2016

\bibitem{Schwabl}
F. Schwabl,
Advanced Quantum Mechanics, Springer Verlag (2005)

\end{thebibliography}
\end{document}